\documentclass[10pt, leqno]{amsart}
\usepackage{amsfonts,amssymb}

\usepackage{amsmath}
\usepackage{amsthm}
\usepackage{amsrefs}
\usepackage{qsymbols}
\usepackage{latexsym}
\usepackage{chngcntr}
\usepackage[noadjust]{cite}
\usepackage{paralist}
\usepackage{esint}

\newtheorem{theorem}{Theorem}[section]
\newtheorem{maintheorem}[theorem]{Main Theorem}
\newtheorem{lemma}[theorem]{Lemma}
\newtheorem{proposition}[theorem]{Proposition}

\theoremstyle{definition}

\newtheorem{remark}[theorem]{Remark}
\newtheorem{observation}[theorem]{Observation}
\newtheorem{assumption}[theorem]{Assumption}

\newcommand{\IR}{\mathbb{R}}
\newcommand{\IC}{\mathbb{C}}
\newcommand{\IN}{\mathbb{N}}

\newcommand{\ID}{\mathbb{D}}

\newcommand{\Lop}{{\mathcal{L}}}

\newcommand{\Sec}{\mathrm{S}}

\newcommand{\T}{\mathcal{T}}
\newcommand{\R}{\mathcal{R}}

\renewcommand{\L}{\mathrm{L}}
\renewcommand{\H}{\mathrm{H}}

\newcommand{\C}{\mathrm{C}}
\newcommand{\W}{\mathrm{W}}

\renewcommand{\d}{\mathrm{d}}

\newcommand{\eps}{\varepsilon}

\newcommand{\abs}[1]{\lvert#1\rvert}

\DeclareMathOperator{\supp}{supp}
\DeclareMathOperator{\dist}{dist}
\DeclareMathOperator{\diam}{diam}

\renewcommand{\Re}{\operatorname{Re}}

\newcommand{\cE}{\mathcal{E}}
\newcommand{\cF}{\mathcal{F}}
\newcommand{\cM}{\mathcal{M}}
\newcommand{\dom}{\mathcal{D}}
\newcommand{\cC}{\mathcal{C}}

\newcommand{\fa}{\mathfrak{a}}

\numberwithin{equation}{section}

\title[$\R$-sectoriality of Higher-order elliptic systems on general bounded domains]{$\R$-sectoriality of higher-order elliptic systems on general bounded domains}
\author{Patrick Tolksdorf}
\address{Fachbereich Mathematik, Technische Universit\"at Darmstadt, Schlossgartenstr. 7, 64289 Darmstadt, Germany}
\email{tolksdorf@mathematik.tu-darmstadt.de}
\thanks{The author was supported by ``Studienstiftung des deutschen Volkes''.}

\date{\today}

\begin{document}
\begin{abstract}
On bounded domains $\Omega \subset \IR^d , d \geq 2$, reaching far beyond the scope of Lipschitz domains, we consider an elliptic system of order $2 m$ in divergence form with complex $\L^{\infty}$-coefficients complemented with homogeneous mixed Dirichlet/Neumann boundary conditions. We prove that the $\L^p$-realization of the corresponding operator $A$ is $\R$-sectorial of angle $\omega \in [0 , \frac{\pi}{2})$, where in the case $2m < d$, $p \in (\frac{2d}{d + 2 m} - \eps , \frac{2d}{d - 2 m} + \eps)$ for some $\eps > 0$, and where $p \in (1 , \infty)$ in the case $2m \geq d$. To perform this proof, we generalize the $\L^p$-extrapolation theorem of Shen to the Banach space valued setting and to arbitrary Lebesgue-measurable underlying sets.
\end{abstract}
\maketitle

\section{Introduction}
\label{Sec: Introduction}

\noindent The main object under consideration is an elliptic operator $A$ in divergence form of order $2m$ formally given by
\begin{align*}
 (A u)_i = (-1)^m \sum_{j = 1}^N \sum_{\abs{\alpha} , \abs{\beta} = m} \partial^{\alpha} [\mu_{\alpha \beta}^{i j} \partial^{\beta} u_j] \qquad (1 \leq i \leq N)
\end{align*}
on a bounded domain $\Omega \subset \IR^d , d \geq 2$. The coefficients $\mu_{\alpha \beta}^{i j}$ are supposed to be essentially bounded and complex valued; ellipticity is enforced by a G\r{a}rding type inequality. Each component of $u$ is supposed to satisfy mixed Dirichlet/Neumann boundary conditions on possibly different portions of the boundary. That is to say, on given closed subsets $D_i \subset \partial \Omega$ all derivatives of order less than $m - 1$ of the $i$th component of $u \in \dom(A)$ are assumed to vanish and on its complement relative to $\partial \Omega$ the $i$th component is assumed to satisfy homogeneous Neumann boundary conditions arising naturally by the definition of the operator. If all $D_i$ coincide a detailed introduction to higher-order elliptic operators and a discussion of these Neumann boundary conditions is included in Brewster, D.\@ Mitrea, I.\@ Mitrea, and M.\@ Mitrea~\cite[Sec.~7]{Brewster_Mitrea_Mitrea_Mitrea}. \par The given boundary conditions have an impact on the admissible geometric constellation of $\partial \Omega$, namely every point in $\overline{\partial \Omega \setminus [\cap_{i = 1}^N D_i]}$ is assumed to possess a bi-Lipschitzian coordinate chart. We record that the intersection of the sets $D_i$ is free from further assumptions and emphasize that the results of this article include the pure Dirichlet and Neumann cases, so that in the first case it suffices to assume the sole openness of $\Omega$ and in the second case that $\Omega$ is a Lipschitz domain. As usual, we interpret $A$ in a weak sense as a sectorial operator on $\L^2(\Omega ; \IC^N)$, i.e., its spectrum is contained in the closure of a sector $\Sec_{\omega} := \{z \in \IC \setminus \{0\} : \abs{\arg(z)} < \omega\}$ for some $\omega \in (0 , \frac{\pi}{2})$ and $\{ \lambda (\lambda + A)^{-1} \}_{\lambda \in \Sec_{\pi - \theta}}$ is bounded for all $\theta \in (\omega , \pi)$. \par
The easiest way to introduce $\R$-sectoriality for operators on $\L^p$-spaces may be the following. A linear operator $B$ on $\L^p$ is $\R$-sectorial of angle $\omega$ if $B$ is sectorial of angle $\omega$ and if for all $\theta \in (\omega , \pi)$ there exists a constant $C > 0$ such that for every $n_0 \in \IN$, all $(\lambda_n)_{n = 1}^{n_0} \subset \Sec_{\pi - \theta},$ and all $(f_n)_{n = 1}^{n_0} \subset \L^p$ the \textit{square function estimate}
\begin{align*}
 \Big\| \Big( \sum_{n = 1}^{n_0} \lvert \lambda_n (\lambda_n + B)^{-1} f_n \rvert^2 \Big)^{\frac{1}{2}} \Big\|_{\L^p} \leq C \Big\| \Big( \sum_{n = 1}^{n_0} \lvert f_n \rvert^2 \Big)^{\frac{1}{2}} \Big\|_{\L^p}
\end{align*}
holds true. It follows that on $\L^2$ the notions of sectoriality and $\R$-sectoriality coincide. Therefore, it is the task of this article to extrapolate $\R$-sectoriality of $A$ from $\L^2(\Omega ; \IC^N)$ to $\L^p(\Omega ; \IC^N)$ within the desired range of $p$'s. Due to the $\ell^2$-norm appearing in the square function estimate, this extrapolation requires a Banach space valued version of the $\L^p$-extrapolation theorem of Shen~\cite[Thm.~3.3]{Shen-Riesz}. Additionally, to meet the generality of the underlying domain, we present a proof of Shen's theorem in the Banach space valued setting and on general Lebesgue-measurable sets. Notice that in the smooth setting, $\R$-sectoriality of higher-order elliptic operators is treated in the monograph of Denk, Hieber, and Pr\"uss~\cite{Denk_Hieber_Pruess}. \par
The bridge between $\R$-sectorial operators and PDEs is built by the theorem of Weis~\cite[Thm.~4.2]{Weis}, which proves that $\R$-sectorial operators of angle less than $\pi / 2$ admit maximal parabolic $\L^q$-regularity. The latter notion is eminent in the treatment of nonlinear parabolic problems and was used in numerous occasions, see, e.g., Pr\"uss~\cite{Pruess}, Denk, Saal, and Seiler~\cite{Denk_Saal_Seiler}, or Haller-Dintelmann and Rehberg~\cite{Haller-Dintelmann_Rehberg}. \par
Recently, $\R$-sectoriality of second-order elliptic equations with real coefficients subject to mixed boundary conditions was established by Auscher, Badr, Haller-Dintelmann, and Rehberg in~\cite{Auscher_Badr_Haller-Dintelmann_Rehberg}. However, the authors deduce the $\R$-sectoriality directly via Gaussian estimates of the corresponding semigroup, so that this line of action does not work for systems of equations. A natural substitute of the Gaussian estimates for systems are off-diagonal estimates. Such an approach to the $\R$-sectoriality of second order systems is presented by Egert in~\cite{Egert_Systems}. This article presents a very short and direct proof of $\R$-sectoriality for higher-order elliptic systems with completely different techniques. Here, the only property of the PDE that is used is the $\L^2$-resolvent estimate and Caccioppoli's inequality for the resolvent equation. We emphasize that this inequality is verified in the second-order case in merely half a page, see, e.g., Shen~\cite[Lem.~2.1]{Shen-Elliptic}. \par
This article is organized as follows. In Section~\ref{Sec: Notations, assumptions and preliminary considerations} we give a precise formulation of all considered objects and notions, whereupon we will be able to state the main result in Section~\ref{Sec: The main result}. The proof of the main result will occupy the rest of this article. In Section~\ref{Sec: Notations, assumptions and preliminary considerations} we also observe that $\R$-sectoriality on $\L^p$-spaces is nothing else than the uniform boundedness of a certain family of operators on the vector valued $\L^p$-space $\L^p(\Omega ; \ell^2(\IC^N))$. The $\L^p$-extrapolation theorem of Shen will be generalized in Section~\ref{Sec: A Banach space valued Lp extrapolation theorem}. \par
Finally, in Section~\ref{Sec: Vector valued weak reverse Hoelder estimates} we are concerned with the proof of the vector valued version of the weak reverse H\"older estimates, which are required for the $\L^p$-extrapolation theorem. This is achieved by locally proving a $\IC^{n_0}$-valued Sobolev embedding with involved constant $C$ independent of $n_0$. The proof will be concluded by establishing Caccioppoli's inequality for functions $u$ that locally solve $\lambda u + A u = 0$. This argument heavily bases on Barton's prove of Caccioppoli's inequality in the higher-order case with no lower-order derivatives on the right-hand side except of the zeroth order term, see~\cite{Barton}.

\subsection*{Acknowledgements}
As this paper is part of my PhD-thesis I want to thank my advisor Robert Haller-Dintelmann for his guidance and help. Moreover, I want to thank Moritz Egert for pointing out the proof of Lemma~\ref{Lem: Reverse Hoelder for all balls}.

\section{Notation, assumptions and preliminary considerations}
\label{Sec: Notations, assumptions and preliminary considerations}

Throughout this article the space dimension $d \geq 2$ is fixed. An open and connected subset of $\IR^d$ will be called a \textit{domain}. A ball with center $x$ and radius $r$ is denoted by $B(x , r)$, whereas a cube centered at $x$, with diameter $2r$, and faces parallel to the coordinate axes is denoted by $Q(x , r)$. For a positive constant $\alpha$ denote the dilated balls with same center by $\alpha B$. Integration will always be with respect to the Lebesgue measure and the Lebesgue measure of a measurable set $A$ will be denoted by $\abs{A}$. If $0 < \abs{A} < \infty$ and $f \in \L^1(A),$ denote the average of $f$ on $A$ by $(f)_A := \abs{A}^{-1} \int_A f \; \d x$. For multiindices we employ the common notation. \par
Banach spaces will always be over the complex field. The set of bounded linear operators on a Banach space $X$ is denoted by $\Lop(X)$. The domain of a linear operator $B$ is denoted by $\dom(B)$ and its spectrum by $\sigma(B)$. For $\omega \in [0 , \pi)$ define the sector $\Sec_{\omega}$ as $\{z \in \IC \setminus \{0\} : \abs{\arg(z)} < \omega \}$ if $\omega > 0$ and $\Sec_{\omega} := (0 , \infty)$ if $\omega = 0$. Mostly, we will make use of a generic constant $C > 0$.

\subsection{The geometric setup}

In this article we will assume that $\Omega$ is a bounded domain `admissible' for mixed boundary conditions, which is defined precisely in the following assumption.

\begin{assumption}
\label{Ass: Mixed boundary geometry}
The domain $\Omega \subset \IR^d$ is bounded and there exists a possibly empty, closed set $D \subset \partial \Omega$, such that for every point $x \in \overline{\partial \Omega \setminus D}$ there exists a bi-Lipschitz coordinate chart. More precisely, there exists a number $M \geq 1$ such that for every $x \in \overline{\partial \Omega \setminus D}$ there exists an open neighborhood $U_x \subset \IR^d$ and a bi-Lipschitz homeomorphism $\Phi_x : U_x \to (-1 , 1)^d$ with Lipschitz constants of $\Phi_x$ and $\Phi_x^{-1}$ being bounded by $M$ and fulfilling the mapping properties
\begin{align*}
 \Phi_x (x) &= 0 \\
 \Phi_x (\Omega \cap U_x) &= (-1 , 1)^{d - 1} \times (0 , 1) \\
 \Phi_x (\partial \Omega \cap U_x) &= (-1 , 1)^{d - 1} \times \{ 0 \}.
\end{align*}
\end{assumption}

\begin{remark}
\label{Rem: Mixed boundary geometry}
\begin{enumerate}
 \item For $y \in \partial \Omega \cap U_x$ with $\abs{x - y} \leq 1 / (2 M)$ and $0 < r \leq 1 / 4$ it is easy to see that $Q(\Phi_x (y) , r) \subset (-1 , 1)^d$ holds. Denote the bi-Lipschitzian counterpart of this cube by $U_{y , r} := \Phi_x^{-1} ((Q(\Phi_x (y) , r)))$ and denote its portion in $\Omega$ by $U_{y , r}^+ := U_{y , r} \cap \Omega$. Note that the bi-Lipschitz property of $\Phi_x$ implies that for $0 < s < t \leq 1$
\begin{align*}
(\sqrt{d}M)^{-1} (t - s) r \leq \dist(\partial U_{y , sr} , \partial U_{y , tr}) \leq \frac{M}{\sqrt{d}} (t - s) r
\end{align*}
holds. \label{Item: Lipschitz cylinders (1)}
 \item \label{Item: Lipschitz cylinders (2)}With $y$ and $r$ as in~\eqref{Item: Lipschitz cylinders (1)}, the bi-Lipschitzianity of $\Phi_x$ implies
\begin{align*}
 B(y , r / (M \sqrt{d})) \subset U_{y , r} \subset B(y , M r).
\end{align*}
\end{enumerate}
\end{remark}

For further reference, we record the following proposition dealing with local extensions at the Lipschitz boundary of $\Omega$. The proof of this proposition is an easy reflection argument and is omitted.

\begin{proposition}
\label{Prop: Local extension operators}
Let $\Omega$ be a domain subject to Assumption~\ref{Ass: Mixed boundary geometry}. Let $x \in \overline{\partial \Omega \setminus D}$, and $y$ and $r$ be as in Remark~\ref{Rem: Mixed boundary geometry}~(1). Then there exists a bounded extension operator $\cE_{y , r} : \L^1(U_{y , r}^+) \to \L^1(U_{y , r})$, i.e., $\cE_{y , r} u|_{U_{y , r}^+} = u$, which restricts for all $p \in [1 , \infty)$ to a bounded operator from $\L^p(U_{y , r}^+)$ into $\L^p(U_{y , r})$ and from $\W^{1 , p}(U_{y , r}^+)$ into $\W^{1 , p}(U_{y , r})$. Moreover, on $U_{y , r} \setminus U_{y , r}^+$ the function $\cE_{y , r} u$ is given by $u \circ \psi$, where $\psi := \Phi_x^{-1} \circ \mathfrak{R} \circ \Phi_x$ and $\mathfrak{R}$ is the reflection at the upper half-space boundary. Furthermore, $\psi$ is Lipschitz continuous on $U_{y , r} \setminus U_{y , r}^+$ with Lipschitz constant bounded by $M^2$.
\end{proposition}

\subsection{The spaces}

In this section we will give a brief introduction to the spaces we will be working with. For a Lebesgue-measurable set $\Xi \subset \IR^d$ and $1 \leq p \leq \infty , \L^p(\Xi ; \IC^N)$ are the usual $\IC^N$-valued Lebesgue spaces. If $X$ is a Banach space, denote the Bochner-Lebesgue spaces by $\L^p(\Xi ; X)$. \par
As we deal with mixed boundary conditions, we need Sobolev spaces that are adapted to these boundary conditions. For this purpose let $\Omega \subset \IR^d , d \geq 2$, be open and let $D \subset \partial \Omega$ be closed. Define for $p \in (1 , \infty)$ and $m \in \IN$ the space
\begin{align*}
 \W^{m , p}_{D} (\Omega) := \mathrm{closure}( \{ \varphi|_{\Omega} : \varphi \in \C_c^{\infty} (\IR^d) , \supp(\varphi) \cap D = \emptyset \} , \| \cdot \|_{\W^{m , p} (\Omega)} ).
\end{align*}
The set $D$ is usually referred to as the \textit{Dirichlet part}, because on this portion of the boundary functions and their derivatives up to order $m - 1$ are forced to vanish. A $\IC^N$-valued counterpart of this definition should reflect that it is natural to have different Dirichlet parts in different components of the $\IC^N$-valued function. Thus, for $D_1 , \dots , D_N \subset \partial \Omega$ closed define $\ID := \prod_{i = 1}^N D_i$ and define
\begin{align*}
 \W^{m , p}_{\ID} (\Omega ; \IC^N) := \prod_{i = 1}^N \W^{m , p}_{D_i} (\Omega)
\end{align*}
with the usual product norm
\begin{align*}
 \|u\|_{\W^{m , p}_{\ID} (\Omega ; \IC^N)} := \Big( \sum_{i = 1}^N \|u_i\|_{\W^{m , p}_{D_i} (\Omega)}^p \Big)^{\frac{1}{p}}.
\end{align*}

\begin{remark}
We have that $\W^{m , p}_{\partial \Omega} (\Omega) = \W^{m , p}_0 (\Omega)$ for any open set $\Omega$ and $\W^{m , p}_{\emptyset} (\Omega) = \W^{m , p} (\Omega)$ if there exists a bounded Sobolev extension operator from $\W^{m , p} (\Omega)$ to $\W^{m , p} (\IR^d)$. Note that the spaces $\W^{m , p}_D$ are systematically studied in \cite[Sec.~4]{Brewster_Mitrea_Mitrea_Mitrea} and if $m = 1$ in \cite{Haller-Dintelmann_Jonsson_Knees_Rehberg}.
\end{remark}

If $\Omega$ and $D$ are subject to Assumption~\ref{Ass: Mixed boundary geometry}, notice that for each $x \in \overline{\partial \Omega \setminus D}$ the sets $U_x \cap \Omega$ are $(\eps , \delta)$-domains in the sense of Jones~\cite{Jones}. This follows as the $(\eps , \delta)$-property is preserved under bi-Lipschitz homeomorphisms and since $(-1 , 1)^{d - 1} \times (0 , 1)$ is an $(\eps , \delta)$-domain. For a proof of this fact, we refer to Egert~\cite[Lem.~2.2.20]{Egert}. Using this covering property of the boundary strip near $\overline{\partial \Omega \setminus D}$ by $(\eps , \delta)$-domains, one can use a partition of unity in order to obtain Sobolev extension operators for the Sobolev spaces $\W^{m , p}_D (\Omega)$. Such a construction yielding semi-universal extension operators can be found in~\cite[Thm.~3.9]{Brewster_Mitrea_Mitrea_Mitrea}. More precisely, these authors prove that there exists a linear operator $\cE$ mapping locally integrable functions on $\Omega$ into Lebesgue measurable functions on $\IR^d$, which satisfies $\cE u|_{\Omega} = u$, which is bounded from $\L^p(\Omega)$ into $\L^p(\IR^d)$, and which is bounded from $\W^{k , p}_D (\Omega)$ into $\W^{k , p}_D (\IR^d)$ for all $1 \leq k \leq m$. Moreover, in the present situation the operator norms depend only on $d$, $p$, $M$, and $m$. Using this, we can prove the following proposition.

\begin{proposition}
\label{Prop: Equivalent norms}
Let $\Omega$ and $D$ be subject to Assumption~\ref{Ass: Mixed boundary geometry}. Then, for each $m \in \IN$ and all $p \in (1 , \infty)$ there exists a constant $C > 0$ depending only on $d$, $p$, $M$, and $m$ such that
\begin{align*}
 \| u \|_{\W^{m , p} (\Omega)} \leq C \big[ \| u \|_{\L^p(\Omega)}^p + \|\nabla^m u\|_{\L^p(\Omega)}^p \big]^{1 / p} \qquad (u \in \W^{m , p}_D (\Omega)).
\end{align*}
\end{proposition}

\begin{proof}
We perform an induction on $m$. Note that there is nothing to do in the case $m = 1$ so that we can directly perform the induction step. Thus, assume the validity of the statement above for a fixed number $m \in \IN$. Then there exists a constant $K \geq 1$ such that
\begin{align*}
 \| u \|_{\W^{m + 1 , p} (\Omega)} \leq K \big[ \| u \|_{\L^p (\Omega)}^p + \| \nabla^m u \|_{\L^p (\Omega)}^p + \| \nabla^{m + 1} u \|_{\L^p(\Omega)}^p \big]^{1 / p}.
\end{align*}
Next, use the Gagliardo--Nirenberg inequality~\cite[p.~125]{Nirenberg} to deduce
\begin{align*}
 \| \nabla^m u \|_{\L^p (\Omega)}^p \leq \| \nabla^m \cE u \|_{\L^p (\IR^d)}^p \leq C \| \cE u \|_{\L^p(\IR^d)}^{\frac{p}{m + 1}} \| \nabla^{m + 1} \cE u \|_{\L^p(\IR^d)}^{\frac{p m}{m + 1}}.
\end{align*}
Using the boundedness properties of $\cE$ discussed in the paragraph prior to the proposition shows
\begin{align*}
 \| \nabla^m u \|_{\L^p (\Omega)}^p &\leq C \| u \|_{\L^p(\Omega)}^{\frac{p}{m + 1}} \| u \|_{\W^{m + 1 , p} (\Omega)}^{\frac{p m}{m + 1}}.
\intertext{Employing Young's inequality yields}
 &\leq \frac{K^{p m} C^{m + 1}}{m + 1} \| u \|_{\L^p (\Omega)}^p + \frac{m}{K^p (m + 1)} \| u \|_{\W^{m + 1 , p} (\Omega)}^p.
\intertext{Finally, the induction hypothesis delivers}
 &\leq \frac{K^{p m} C^{m + 1} + m}{m + 1} \| u \|_{\L^p (\Omega)}^p + \frac{m}{m + 1} \| \nabla^m u \|_{\L^p(\Omega)}^p \\ &\quad+ \frac{m}{K^p (m + 1)} \| \nabla^{m + 1} u \|_{\L^p(\Omega)}^p.
\end{align*}
Absorbing the second summand on the right-hand side to the left-hand side concludes the induction step.
\end{proof}

\subsection{The operator}
\label{Subsec: The operator}

Let $N \in \IN$ denote the number of equations of the elliptic system, which itself is supposed to be of order $2m$ with $m \in \IN$. Fix $\Omega \subset \IR^d$ and closed sets $D_1 , \dots , D_N \subset \partial \Omega$ and define
\begin{align*}
 D:= \bigcap_{i = 1}^N D_i \quad \text{and} \quad \ID := \prod_{i = 1}^N D_i.
\end{align*}
Suppose that $\Omega$ and $D$ fulfill Assumption~\ref{Ass: Mixed boundary geometry}. For the coefficients $\mu_{\alpha \beta}^{i j}$ we make the following assumption.

\begin{assumption}
\label{Ass: Gardings inequality}
The coefficients $\mu_{\alpha \beta}^{i j} : \Omega \to \IC , 1 \leq i , j \leq N , \alpha , \beta \in \IN_0^d$ with $\abs{\alpha} = \abs{\beta} = m$ are Lebesgue measurable, bounded functions with bound $\Lambda > 0$ such that the sesquilinear form
\begin{align*}
 \mathfrak{a} : \W^{m , 2}_{\ID} (\Omega ; \IC^N) \times \W^{m , 2}_{\ID} (\Omega ; \IC^N), \quad (u , v) \mapsto \sum_{i , j = 1}^N \sum_{\abs{\alpha} , \abs{\beta} = m} \int_{\Omega} \mu_{\alpha \beta}^{i j} \partial^{\beta} u_j \overline{\partial^{\alpha} v_i} \; \d x
\end{align*}
is \textit{elliptic} in the sense that for some $\kappa > 0$ it satisfies the \textit{G\r{a}rding inequality}
\begin{align*}
 \Re(\mathfrak{a}(u , u)) \geq \kappa \sum_{i = 1}^N \sum_{\abs{\alpha} = m} \int_{\Omega} \abs{\partial^{\alpha} u_i}^2 \; \d x = \kappa \|\nabla^m u\|_{\L^2(\Omega ; \IC^{N d^m})}^2 \qquad (u \in \W^{m , 2}_{\ID} (\Omega ; \IC^N)).
\end{align*}
\end{assumption}

\begin{remark}
\label{Rem: Numerical range}
Under Assumption~\ref{Ass: Gardings inequality} the sesquilinear form $\mathfrak{a}$ is sectorial of an angle $\omega \in [0 , \frac{\pi}{2})$, i.e., the \textit{numerical range}
\begin{align*}
 \{ \mathfrak{a}(u , u) : u \in \W^{m , 2}_{\ID} (\Omega ; \IC^N) \}
\end{align*}
is contained in a sector $\overline{\Sec_{\omega}}$.
\end{remark}

Since $\mathfrak{a}$ is densely defined, sectorial, and closed (this follows by ellipticity of $\mathfrak{a}$ and Proposition~\ref{Prop: Equivalent norms}), it is known from classical form theory \cite[Thm.~VI.2.1]{Kato} that there exists a unique sectorial operator $A$ on $\L^2(\Omega ; \IC^N)$ of angle $\omega \in [0 , \frac{\pi}{2})$ such that $\dom(A) \subset \W^{m , 2}_{\ID} (\Omega ; \IC^N)$ and
\begin{align*}
 \mathfrak{a} (u , v) = (A u , v)_{\L^2} \qquad (u \in \dom(A) , v \in \W^{m , 2}_{\ID}(\Omega ; \IC^N)),
\end{align*}
where $(\cdot , \cdot)_{\L^2}$ denotes the $\L^2$ inner product. Here, we say that a closed linear operator $B : \dom(B) \subset X \to X$ is \textit{sectorial of angle $\omega \in [0 , \pi)$} if $\sigma(B) \subset \overline{\Sec_{\omega}}$ and if for every $\theta \in (\omega , \pi]$ there exists a constant $C > 0$ with
\begin{align*}
 \|\lambda (\lambda + B)^{-1}\|_{\Lop(X)} \leq C \qquad (\lambda \in \Sec_{\pi - \theta}).
\end{align*}
\indent If $B$ is a linear operator on $\L^2(\Xi ; \IC^N)$ on a bounded measurable set $\Xi \subset \IR^d$ and $p > 2$ define the $\L^p$-realization of $B_p$ as the \textit{part} of $B$ in $\L^p$, i.e.,
\begin{align*}
 \dom(B_p) &:= \{u \in \dom(B) \cap \L^p(\Xi ; \IC^N) : B u \in \L^p(\Xi ; \IC^N)\}, \\
 B_p u &:= B u \qquad (u \in \dom(B_p)).
\end{align*}
If $p^{\prime} < 2$ define $B_{p^{\prime}}$ as the closure of $B$ in $\L^{p^{\prime}}(\Xi ; \IC^N)$, if it exists. We record the following lemma, which connects $B_{p^{\prime}}$ and $(B^*)_p$ for $1 / p + 1 / p^{\prime} = 1$.

\begin{lemma}
\label{Lem: Lp realization}
Let $\Xi \subset \IR^d$ be a bounded domain, $p \in (2 , \infty)$, $1 / p + 1 / p^{\prime} = 1,$ and let $B$ be a densely defined operator on $\L^2(\Xi ; \IC^N)$. Then $\dom(B)$ is dense in $\L^{p^{\prime}}(\Xi ; \IC^N)$ and $B$ is closable in $\L^{p^{\prime}}(\Xi ; \IC^N)$ if and only if the part of $B^*$ in $\L^p (\Xi ; \IC^N)$ is densely defined. In this case the identity $(B_{p^{\prime}})^* = (B^*)_p$ holds true.
\end{lemma}

\begin{proof}
First, assume there exists $f \in \L^p(\Xi ; \IC^N)$ such that
\begin{align*}
 \int_{\Xi} u \cdot \overline{f} \; \d x = 0 \qquad (u \in \dom(B)).
\end{align*}
By the boundedness of $\Xi$, we find $f \in \L^2(\Xi ; \IC^N)$, and by the density of $\dom(B)$ in $\L^2(\Xi ; \IC^N)$, it follows that $f$ must be zero. Consequently, $\dom(B)$ is dense in $\L^{p^{\prime}} (\Xi ; \IC^N)$. \par
Let $B$ be closable in $\L^{p^{\prime}} (\Xi ; \IC^N)$. Because $\dom(B) \subset \dom(B_{p^{\prime}})$ the closure of $B$ is densely defined and by definition of its domain, for each $u \in \dom(B_{p^{\prime}})$ there exists a sequence $(u_n)_{n \in \IN} \subset \dom(B)$ with $u_n \to u$ in $\L^{p^{\prime}} (\Xi ; \IC^N)$ and $B u_n \to B_{p^{\prime}} u$ in $\L^{p^{\prime}} (\Xi ; \IC^N)$. Thus, for $v \in \dom((B^*)_p)$, we find
\begin{align*}
 \langle B_{p^{\prime}} u , v \rangle_{\L^{p^{\prime}} , \L^p} = \lim_{n \to \infty} ( u_n , B^* v )_{\L^2} = \langle u , (B^*)_p v \rangle_{\L^{p^{\prime}} , \L^p},
\end{align*}
where $\langle \cdot , \cdot \rangle_{\L^{p^{\prime}} , \L^p}$ denotes the duality product between $\L^{p^{\prime}} (\Xi ; \IC^N)$ and $\L^p (\Xi ; \IC^N)$. We derive the inclusion $\dom((B^*)_p) \subset \dom((B_{p^{\prime}})^*)$ and equality of the operators on $\dom((B^*)_p)$. If $w \in \dom((B_{p^{\prime}})^*)$, we find for $u \in \dom(B) \subset \dom(B_{p^{\prime}})$
\begin{align*}
 ( B u , w )_{\L^2} = \langle u , (B_{p^{\prime}})^* w \rangle_{\L^{p^{\prime}} , \L^p} = ( u , (B_{p^{\prime}})^* w )_{\L^2}.
\end{align*}
Consequently, $w \in \dom(B^*) \cap \L^p(\Xi ; \IC^N)$ and $B^* w = (B_{p^{\prime}})^* w \in \L^p(\Xi ; \IC^N)$ so that $w \in \dom((B^*)_p)$. This proves $\dom((B^*)_p) = \dom((B_{p^{\prime}})^*)$, so that the part of $B^*$ in $\L^p (\Xi ; \IC^N)$ is densely defined by Schechter~\cite[Thm.~7.20 \& Lem.~7.21]{Schechter}. \par
Now assume that $(B^*)_p$ is densely defined and let $(u_n)_{n \in \IN} \subset \dom(B)$ be a sequence with $u_n \to 0$ in $\L^{p^{\prime}} (\Xi ; \IC^N)$ and such that $(B u_n)_{n \in \IN}$ is a Cauchy sequence in $\L^{p^{\prime}} (\Xi ; \IC^N)$ with limit $f$. Then
\begin{align*}
 \langle f , v \rangle_{\L^{p^{\prime}} , \L^p} = \lim_{n \to \infty} \langle u_n , (B^*)_p v \rangle_{\L^{p^{\prime}} , \L^p} = 0 \qquad (v \in \dom((B^*)_p)).
\end{align*}
Hence, $f$ is zero by density of $\dom((B^*)_p)$ in $\L^p(\Xi ; \IC^N)$ so that $B$ is closable in $\L^{p^{\prime}} (\Xi ; \IC^N)$.
\end{proof}

\subsection{$\R$-sectoriality}
\label{Subsec: R-sectoriality}

For a Lebesgue-measurable set $\Xi \subset \IR^d , p \in [1 , \infty),$ and $N \in \IN$, a linear operator $B : \dom(B) \subset \L^p(\Xi ; \IC^N) \to \L^p(\Xi ; \IC^N)$ is called $\R$-sectorial of angle $\omega$ if $B$ is sectorial of angle $\omega$ and if additionally for all $\theta \in (\omega , \pi]$ there exists a constant $C > 0$ such that for all $n_0 \in \IN , (\lambda_n)_{n = 1}^{n_0} \subset \Sec_{\pi - \theta}$, and $(f_n)_{n = 1}^{n_0} \subset \L^p(\Xi ; \IC^N)$ the \textit{square function estimate}
\begin{align}
 \label{Eq: Characterization square function estimate}
\Big\| \Big[ \sum_{n = 1}^{n_0} \abs{\lambda_n (\lambda_n + B)^{-1} f_n}^2 \Big]^{\frac{1}{2}} \Big\|_{\L^p(\Xi)} \leq C \Big\| \Big[ \sum_{n = 1}^{n_0} \abs{f_n}^2 \Big]^{\frac{1}{2}} \Big\|_{\L^p(\Xi)}
\end{align}
holds true.

\begin{remark}
For the general notion of $\R$-sectoriality for operators acting on general Banach spaces beyond the scope of $\L^p$, see, e.g., Denk, Hieber, and Pr\"uss~\cite[Sec.~4.1]{Denk_Hieber_Pruess}. It is a matter of fact, that on $\L^p$-spaces the definition of $\R$-sectoriality via square function estimates is equivalent to the one given in~\cite[Sec.~4.1]{Denk_Hieber_Pruess}, see Kunstmann and Weis~\cite[Rem.~2.9]{Kunstmann-Weis}.
\end{remark}

Recall that sectorial operators on reflexive Banach spaces are always densely defined, see Haase~\cite[Prop.~2.1.1 h)]{Haase}. Thus, an application of Kalton and Weis~\cite[Lem.~3.1]{Kalton_Weis} delivers the following dualization principle.

\begin{observation}
\label{Ob: Duality}
If $p \in (1 , \infty)$, then $B$ is $\R$-sectorial on $\L^p(\Xi ; \IC^N)$ if and only if its adjoint $B^*$ is $\R$-sectorial on $\L^{p^{\prime}} (\Xi ; \IC^N)$, where $1 / p + 1 / p^{\prime} = 1$.
\end{observation}

A closer look on the square function estimate reveals that~\eqref{Eq: Characterization square function estimate} is nothing else than the uniform boundedness estimate of the operators
\begin{align*}
 (T_{\lambda_1} , \dots , T_{\lambda_{n_0}} , 0 , \dots) : \L^p(\Xi ; \ell^2(\IC^N)) &\to \L^p(\Xi ; \ell^2(\IC^N)), \\
 (f_n)_{n \in \IN} &\mapsto (T_{\lambda_1} f_1 , \dots , T_{\lambda_{n_0}} f_{n_0} , 0 , \dots),
\end{align*}
where $T_{\lambda} := \lambda (\lambda + B)^{-1}$ and $\ell^2 (\IC^N)$ denotes the Banach space of square summable $\IC^N$-valued sequences. Thus, the following proposition is evident.

\begin{proposition}
\label{Prop: Reinterpretation of R-boundedness}
A linear operator $B$ on $\L^p(\Xi ; \IC^N)$ is $\R$-sectorial of angle $\omega$ if and only if $\sigma(B) \subset \overline{\Sec_{\omega}}$ and if for every $\theta \in (\omega , \pi]$ the family
\begin{align}
\label{Eq: Operator family to be extended}
 \{ (T_{\lambda_1} , \dots , T_{\lambda_{n_0}} , 0 , \dots) : n_0 \in \IN , (\lambda_n)_{n = 1}^{n_0} \subset \Sec_{\pi - \theta} \}
\end{align}
is bounded in $\Lop(\L^p(\Xi ; \ell^2(\IC^N)))$, where $T_{\lambda} := \lambda (\lambda + B)^{-1}$.
\end{proposition}

\section{The main result}
\label{Sec: The main result}

\noindent We are now in the position to state our main result. Again $\omega \in [0 , \frac{\pi}{2})$ denotes the angle of sectoriality of $A$ on $\L^2(\Omega ; \IC^N)$, see Subsection~\ref{Subsec: The operator}.

\begin{maintheorem}
\label{Thm: Maintheorem}
Let $N , m \in \IN$, $\Omega \subset \IR^d$ be a bounded domain, and $D_1 , \dots , D_N \subset \partial \Omega$ be closed sets. Define
\begin{align*}
 D := \bigcap_{i = 1}^N D_i
\end{align*}
and assume that $\Omega$ and $D$ fulfill Assumption~\ref{Ass: Mixed boundary geometry}. Let the coefficients $\mu_{\alpha \beta}^{i j}$, $1 \leq i , j \leq N$, $\abs{\alpha} = \abs{\beta} = m$ be subject to Assumption~\ref{Ass: Gardings inequality}, and let $A$ be the elliptic operator of order $2m$ as defined in Subsection~\ref{Subsec: The operator}. \par
Then for each $\theta \in [\omega , \pi]$ there exists $\eps \geq 0$ with $\eps > 0$ if $\theta \in (\omega , \pi]$, depending only on $d$, $M$, $m$, $N$, $\kappa$, $\theta$, $\omega$, and $\Lambda$, such that the following statement is valid. \par
If $2m < d$, then for all numbers $p$ satisfying
\begin{align*}
 \frac{2d}{d + 2m} - \eps < p < \frac{2d}{d - 2m} + \eps
\end{align*}
or, if $2m \geq d$, then for all $p \in (1 , \infty)$, the $\L^p$-realization of $A$ is closed and densely defined. Moreover, $A_p$ is sectorial of angle $\theta$ and for every $\theta^{\prime} \in (\theta , \pi]$ the family $\{ \lambda (\lambda + A_p)^{-1} \}_{\lambda \in \Sec_{\pi - \theta^{\prime}}}$ is $\R$-bounded in $\Lop(\L^p(\Omega ; \IC^N))$. 
\end{maintheorem}

\begin{remark}
\label{Rem: Reduction to p > 2}
To prove this theorem, we can reduce matters to the case $p > 2$. Indeed, note that Assumption~\ref{Ass: Gardings inequality} on the coefficients is stable under the operation $\mu_{\alpha \beta}^{i j} \mapsto \overline{\mu_{\beta \alpha}^{j i}}$ so that if the theorem is proven under this assumption for the $\L^p$-realization of $A$ and $p > 2$, it then is also proven for the $\L^p$-realization of $A^*$. For the situation of $p < 2$ one can argue by duality using Observation~\ref{Ob: Duality} and Lemma~\ref{Lem: Lp realization}.
\end{remark}

By sectoriality of $A$ on $\L^2(\Omega ; \IC^N)$ the set defined in~\eqref{Eq: Operator family to be extended} is bounded in $\Lop(\L^2(\Omega ; \ell^2(\IC^N)))$. Thus, with Remark~\ref{Rem: Reduction to p > 2} in mind, it is desirable to provide a tool to extrapolate bounded operators on $\L^2(\Omega ; \ell^2(\IC^N))$ to $\L^p(\Omega ; \ell^2(\IC^N))$ for $p > 2$. This is what we do in the following section.

\section{A Banach space valued $\L^p$-extrapolation theorem}
\label{Sec: A Banach space valued Lp extrapolation theorem}

\noindent The following theorem generalizes the $\L^p$-extrapolation theorem of Shen~\cite[Thm.~3.3]{Shen-Riesz} in two directions. The first is that the extrapolation theorem remains valid in the Banach space valued setting, which is, in view of Proposition~\ref{Prop: Reinterpretation of R-boundedness}, important for $\R$-boundedness. The second is, that it proves the extrapolation theorem far beyond the scope of bounded Lipschitz domains as it was established in~\cite{Shen-Riesz}; here, the requirement is the sole measurability of the underlying domain.

\begin{theorem}
\label{Thm: Extrapolation of square function estimates}
Let $X$ and $Y$ be Banach spaces, $\Omega \subset \IR^d$ be Lebesgue-measurable, $\cM > 0$, and let $T \in \Lop(\L^2(\Omega ; X) , \L^2(\Omega ; Y))$ with $\| T \|_{\Lop(\L^2(\Omega ; X) , \L^2(\Omega ; Y))} \leq \cM$. \par
Suppose that there exist constants $p > 2$, $R_0 > 0$, $\alpha_2 > \alpha_1 > 1$, and $\cC > 0$, where $R_0 = \infty$ if $\diam(\Omega) = \infty$, such that the following holds. For all $B = B(x_0 , r)$ with $0 < r < R_0$, which are either centered on $\partial \Omega$, i.e., $x_0 \in \partial \Omega$, or satisfy $\alpha_2 B \subset \Omega$, and all compactly supported $f \in \L^{\infty}(\Omega ; X)$ with $f = 0$ on $\Omega \cap \alpha_2 B$ the estimate
\begin{align}
\label{Eq: Weak reverse Hoelder inequality}
 \begin{aligned}
 \bigg( \frac{1}{r^d} \int_{\Omega \cap B} \| T f \|_Y^p \; \d x \bigg)^{\frac{1}{p}} &\leq \cC \bigg\{ \bigg( \frac{1}{r^d} \int_{\Omega \cap \alpha_1 B} \| T f \|_Y^2 \; \d x \bigg)^{\frac{1}{2}} \\
 &\qquad+ \sup_{B^{\prime} \supset B} \bigg( \frac{1}{\lvert B^{\prime} \rvert} \int_{\Omega \cap B^{\prime}} \| f \|_X^2 \; \d x \bigg)^{\frac{1}{2}} \bigg\}
 \end{aligned}
\end{align}
holds. Here the supremum runs over all balls $B^{\prime}$ containing $B$. \par
Then for each $2 < q < p$ the restriction of $T$ onto $\L^2(\Omega ; X) \cap \L^q(\Omega ; X)$ extends to a bounded linear operator from $\L^q (\Omega ; X)$ into $\L^q(\Omega ; Y)$, with operator norm bounded by a constant depending on $d$, $p$, $q$, $\alpha_1$, $\alpha_2$, $\cC$, and $\cM$, and additionally on $R_0$ and $\diam(\Omega)$ if $\Omega$ is bounded.
\end{theorem}

The proof of the theorem above is roughly as follows. First, we will take the Banach space valued case with $\Omega = \IR^d$ as granted. The validity of this case was already observed by Auscher in the remark below~\cite[Thm.~1.2]{Auscher} and follows directly from Shen's proof of the whole space case~\cite[Thm.~3.1]{Shen-Riesz} by replacing the absolute value signs by the respective Banach space norms. As Shen's proof is presented very clearly, we omit further details of the Banach space valued whole space case. \par
In the case, where $\Omega$ is not the whole space, one can reduce matters to the whole space case by considering the operator
\begin{align}
\label{Eq: TRd}
 T_{\IR^d} f := \mathcal{E}_0 T \mathcal{R}_{\Omega} f \qquad (f \in \L^2(\IR^d ; X)),
\end{align}
where $\mathcal{R}_{\Omega}$ restricts functions from $\IR^d$ to $\Omega$ and $\mathcal{E}_0$ extends functions from $\Omega$ to $\IR^d$ by zero. Obviously, if the restriction of $T$ onto $\L^2(\Omega ; X) \cap \L^q(\Omega ; X)$ extends to a bounded linear operator from $\L^q (\Omega ; X)$ into $\L^q(\Omega ; Y)$ the same is valid for $T_{\IR^d}$ (with $\Omega$ replaced by $\IR^d$). On the other hand, if the restriction of $T_{\IR^d}$ onto $\L^2(\IR^d ; X) \cap \L^q(\IR^d ; X)$ extends to a bounded operator from $\L^q(\IR^d ; X)$ into $\L^q(\IR^d ; Y)$, then the same is valid for $T$ as well (with $\IR^d$ replaced by $\Omega$). This is true as $T$ can be written as
\begin{align*}
 T f = \R_{\Omega} T_{\IR^d} \mathcal{E}_0 f \qquad (f \in \L^2(\Omega ; X)).
\end{align*}
Consequently, we are left with proving the boundedness of $T_{\IR^d}$. Comparing the assumptions of the cases $\Omega = \IR^d$ and $\Omega \neq \IR^d$ in Theorem~\ref{Thm: Extrapolation of square function estimates} we see that $T$ verifies~\eqref{Eq: Weak reverse Hoelder inequality} only for balls that are either centered on the boundary or lie completely inside $\Omega$ (with some safety distance). However, one has to verify~\eqref{Eq: Weak reverse Hoelder inequality} for $T_{\IR^d}$ for all balls in $\IR^d$. This bridge is built by the following two lemmas. The purpose of the first lemma is to show that~\eqref{Eq: Weak reverse Hoelder inequality} is even valid for all balls that have a non-trivial intersection with $\Omega$ but whose radius is still restricted by the number $R_0$ of Theorem~\ref{Thm: Extrapolation of square function estimates}. The purpose of the second lemma is to show that one can replace the number $R_0$ by an arbitrary other number $R_0^{\prime}$ (if $R_0$ is finite).

\begin{lemma}
\label{Lem: Reverse Hoelder for all balls}
Let $\Omega \subset \IR^d$ be Lebesgue-measurable, $f , g \in \L^2(\Omega)$, $\alpha_2 > \alpha_1 > 1$, $p > 2$, and $r > 0$ and $x_0 \in \IR^d$ be such that $B(x_0 , r) \cap \Omega \neq \emptyset$. If there exists $C > 0$ such that
\begin{align*}
 \bigg( \frac{1}{r^d} \int_{\Omega \cap \widetilde{B}} \abs{f}^p \; \d x \bigg)^{\frac{1}{p}} \leq C \bigg\{ \bigg( \frac{1}{r^d} \int_{\Omega \cap \alpha_1 \widetilde{B}} \abs{f}^2 \; \d x \bigg)^{\frac{1}{2}} + \sup_{B^{\prime} \supset \widetilde{B}} \bigg( \frac{1}{\lvert B^{\prime} \rvert} \int_{\Omega \cap B^{\prime}} \abs{g}^2 \; \d x \bigg)^{\frac{1}{2}} \bigg\}
\end{align*}
holds for all balls $\widetilde{B}$ with $\alpha_2 \widetilde{B} \subset B(x_0 , \alpha_2 r)$ and which are either centered on $\partial \Omega$ or satisfy $\alpha_2 \widetilde{B} \subset \Omega$, then, for each $\alpha \in (1 , \alpha_2)$ there exists a constant $C^{\prime}$ such that
\begin{align*}
 \bigg( \frac{1}{r^d} \int_{\Omega \cap B(x_0 , r)} \abs{f}^p \; \d x \bigg)^{\frac{1}{p}} \leq C^{\prime} \bigg\{ \bigg( \frac{1}{r^d} \int_{\Omega \cap B(x_0 , \alpha r)} \abs{f}^2 \; \d x \bigg)^{\frac{1}{2}} + \sup_{B^{\prime} \supset B(x_0 , r)} \bigg( \frac{1}{\lvert B^{\prime} \rvert} \int_{\Omega \cap B^{\prime}} \abs{g}^2 \; \d x \bigg)^{\frac{1}{2}} \bigg\},
\end{align*}
where $C^{\prime}$ depends on $d$, $\alpha$, $\alpha_1$, $\alpha_2$, $p$, and $C$.
\end{lemma}

\begin{proof}
Define
\begin{align*}
c := \min\bigg\{ \frac{\alpha_2 - 1}{5 \alpha_2 + 1} , \frac{\alpha - 1}{5 \alpha_1 + 1} \bigg\}
\end{align*}
and
\begin{align*}
 \mathcal{I}_1 &:= \{y \in \Omega \cap B(x_0 , r) : B(y , c r) \subset \Omega \}, \\
 \mathcal{I}_2 &:= \{y^{\prime} \in \partial \Omega : \text{there is } y \in \Omega \cap B(x_0 , r) \text{ such that } y^{\prime} \in B(y , c r ) \}.
\end{align*}
Note that for all $y \in \mathcal{I} := \mathcal{I}_1 \cup \mathcal{I}_2$ we have $B(y , 5 c \alpha_1 r) \subset B(x_0 , \alpha r)$ and $B(y , 5 c \alpha_2 r) \subset B(x_0 , \alpha_2 r)$ by definition of $c$. Moreover, by construction
\begin{align*}
 \Omega \cap B(x_0 , r) \subset \Omega \cap \bigcup_{y \in \mathcal{I}} B(y , c r) .
\end{align*}
The covering lemma of Vitali, see Evans and Gariepy~\cite[Thm.~1.5.1]{Evans_Gariepy}, yields an at most countable index set $\mathcal{F}$ such that all balls in the family $\{B(y , c r)\}_{y \in \mathcal{F}}$ are pairwise disjoint and such that
\begin{align*}
 \bigcup_{y \in \mathcal{I}} B(y , c r) \subset \bigcup_{y \in \mathcal{F}} B(y , 5 c r).
\end{align*}
Furthermore, for $\sharp(\cF)$ being the number of points in $\cF$, we have
\begin{align*}
 \abs{B(0 , 1)} c^d r^d \sharp(\cF) = \sum_{y \in \mathcal{F}} \lvert B(y , c r) \rvert \leq \lvert B(x_0 , \alpha r) \rvert = \abs{B(0 , 1)} \alpha^d r^d,
\end{align*}
that is $\sharp(\cF) \leq (\alpha / c)^d$. This yields by hypothesis
\begin{align*}
 \frac{1}{r^d} \int_{\Omega \cap B(x_0 , r)} \lvert f \rvert^p \; \d x &\leq \frac{1}{r^d} \sum_{y \in \mathcal{F}} \int_{\Omega \cap B(y , 5cr)} \lvert f \rvert^p \; \d x \\
 &\leq C^p \sum_{y \in \mathcal{F}} \bigg\{ \bigg( \frac{1}{r^d} \int_{\Omega \cap B(y , 5 \alpha_1 c r)} \lvert f \rvert^2 \; \d x \bigg)^{\frac{1}{2}} \\
 &\qquad+ \sup_{B^{\prime} \supset B(y , 5 c r)} \bigg( \frac{1}{\lvert B^{\prime} \rvert} \int_{\Omega \cap B^{\prime}} \lvert g \rvert^2 \; \d x \bigg)^{\frac{1}{2}} \bigg\}^p.
 \intertext{Define $\beta := (2 + c)/(5 c)$ and note that $\beta \geq 1$. Using this together with $B(y , 5 c \alpha_1 r) \subset B(x_0 , \alpha r)$ delivers}
 &\leq \beta^{\frac{dp}{2}} C^p \sum_{y \in \mathcal{F}} \bigg\{ \bigg( \frac{1}{r^d} \int_{\Omega \cap B(x_0 , \alpha r)} \lvert f \rvert^2 \; \d x \bigg)^{\frac{1}{2}} \\
 &\qquad+ \sup_{\beta B^{\prime} \supset B(y , 5 c \beta r)} \bigg( \frac{1}{\lvert \beta B^{\prime} \rvert} \int_{\Omega \cap \beta B^{\prime}} \lvert g \rvert^2 \; \d x \bigg)^{\frac{1}{2}} \bigg\}^p.
\intertext{Finally, the choice of $\beta$ ensures $B(x_0 , r) \subset B(y , 5 c \beta r)$ for all $y \in \mathcal{F}$. Thus, the supremum becomes larger if we replace $\beta B^{\prime}$ by arbitrary balls that contain $B(x_0 , r)$. This implies}
 &\leq \beta^{\frac{dp}{2}} C^p \sharp(\mathcal{F}) \bigg\{ \bigg( \frac{1}{r^d} \int_{\Omega \cap B(x_0 , \alpha r)} \lvert f \rvert^2 \; \d x \bigg)^{\frac{1}{2}} \\
 &\qquad+ \sup_{B^{\prime} \supset B(x_0 , r)} \bigg( \frac{1}{\lvert B^{\prime} \rvert} \int_{\Omega \cap B^{\prime}} \lvert g \rvert^2 \; \d x \bigg)^{\frac{1}{2}} \bigg\}^p
\end{align*}
and concludes the proof.
\end{proof}

\begin{lemma}
\label{Lem: Bridge from small balls to large balls}
Let $R_0^{\prime} > R_0 > 0$, $f , g \in \L^2 (\IR^d)$, $\alpha > 1$, and $p > 2$. If there exists a constant $C > 0$ such that for all $x_0 \in \IR^d$ and $0 < r < R_0$ the inequality
\begin{align*}
 \bigg( \frac{1}{r^d} \int_{B(x_0 , r)} \abs{f}^p \; \d x \bigg)^{\frac{1}{p}} &\leq C \bigg\{ \bigg( \frac{1}{r^d} \int_{B(x_0 , \alpha r)} \abs{f}^2 \; \d x \bigg)^{\frac{1}{2}} + \sup_{B^{\prime} \supset B(x_0 , r)} \bigg( \frac{1}{\lvert B^{\prime} \rvert} \int_{B^{\prime}} \abs{g}^2 \; \d x \bigg)^{\frac{1}{2}} \bigg\}
\end{align*}
holds. Then there exists a constant $C^{\prime}$ depending on $C$, $\alpha$, $R_0$, $R_0^{\prime}$, and $d$ such that the same inequality holds for all $0 < r < R_0^{\prime}$ with $C$ replaced by $C^{\prime}$.
\end{lemma}

\begin{proof}
Let $x_0 \in \IR^d$, $r \in [R_0 , R_0^{\prime})$, and $B := B(x_0 , r)$. Define
\begin{align*}
 \beta := \min\Big\{ \frac{\alpha - 1}{5 \alpha} , \frac{R_0}{5 R_0^{\prime}} \Big\}.
\end{align*}
It is clear that $\{ B(y , \beta r) \}_{y \in B}$ covers $B$. The covering lemma of Vitali yields an at most countable subset $\cF \subset B$ such that the balls $\{ B(y , \beta r) \}_{y \in \cF}$ are pairwise disjoint and such that
\begin{align*}
 B \subset \bigcup_{y \in \cF} B(y , 5 \beta r).
\end{align*}
Furthermore, since $\alpha > 1$ we find $\beta \leq \alpha - 1$ and conclude that $B(y , \beta r) \subset \alpha B$ for each $y \in \cF$. Consequently,
\begin{align*}
 \abs{B(0 , 1)} \beta^d r^d \sharp (\cF) = \sum_{y \in \cF} \abs{B(y , \beta r)} \leq \abs{B(0 , 1)} \alpha^d r^d
\end{align*}
and thus $\sharp(\cF) \leq (\alpha / \beta)^d$. Now, using the covering property of $\{ B(y , 5 \beta r) \}_{y \in \cF}$ in the first and $5 \beta r < R_0$ in the second inequality yields
\begin{align*}
 \bigg( \frac{1}{r^d} \int_B \abs{f}^p \; \d x \bigg)^{\frac{1}{p}} &\leq [5 \beta]^{\frac{d}{p}} \sum_{y \in \cF} \bigg( \frac{1}{[5 \beta r]^d} \int_{B(y , 5 \beta r)} \abs{f}^p \; \d x \bigg)^{\frac{1}{p}} \\
 &\leq [5 \beta]^{\frac{d}{p}} C \sum_{y \in \cF} \bigg\{ \bigg( \frac{1}{[5 \beta r]^d} \int_{B(y , 5 \alpha \beta r)} \abs{f}^2 \; \d x \bigg)^{\frac{1}{2}} \\
 &\qquad+ \sup_{B^{\prime} \supset B(y , 5 \beta r)} \bigg( \frac{1}{\abs{B^{\prime}}} \int_{B^{\prime}} \abs{g}^2 \; \d x \bigg)^{\frac{1}{2}} \bigg\}.
\end{align*}
Next, $5 \alpha \beta \leq \alpha - 1$ and $B(y , (\alpha - 1) r) \subset \alpha B$ imply that the first integral on the right-hand side is controlled by
\begin{align*}
 \bigg( \frac{1}{[5 \beta r]^d} \int_{B(y , 5 \alpha \beta r)} \abs{f}^2 \; \d x \bigg)^{\frac{1}{2}} \leq \frac{1}{[5 \beta]^{\frac{d}{2}}} \bigg( \frac{1}{r^d} \int_{\alpha B} \abs{f}^2 \; \d x \bigg)^{\frac{1}{2}}.
\end{align*}
For the supremum, we first use that $5 \beta \leq 2$ and then, that the arising averages are taken solely on balls that contain $B(y , 2 r)$. Since $B \subset B(y , 2r)$ for every $y \in \cF$, the supremum will be larger if it runs over all balls that contain $B$. Indeed,
\begin{align*}
 \sup_{B^{\prime} \supset B(y , 5 \beta r)} \bigg( \frac{1}{\abs{B^{\prime}}} \int_{B^{\prime}} \abs{g}^2 \; \d x \bigg)^{\frac{1}{2}} &\leq \Big( \frac{2}{5 \beta} \Big)^{\frac{d}{2}} \sup_{B^{\prime} \supset B(y , 5 \beta r)} \bigg( \frac{1}{\lvert \frac{2}{5 \beta} B^{\prime} \rvert} \int_{\frac{2}{5 \beta} B^{\prime}} \abs{g}^2 \; \d x \bigg)^{\frac{1}{2}} \\
 &\leq \Big( \frac{2}{5 \beta} \Big)^{\frac{d}{2}} \sup_{B^{\prime} \supset B} \bigg( \frac{1}{\lvert B^{\prime} \rvert} \int_{B^{\prime}} \abs{g}^2 \; \d x \bigg)^{\frac{1}{2}}.
\end{align*}
We conclude the proof by recalling the bound on $\sharp(\cF)$.
\end{proof}

\begin{proof}[Proof of Theorem~\ref{Thm: Extrapolation of square function estimates}]
Recalling the discussion between Theorem~\ref{Thm: Extrapolation of square function estimates} and Lemma~\ref{Lem: Reverse Hoelder for all balls} it suffices to consider the operator $T_{\IR^d}$ defined in~\eqref{Eq: TRd}. By Lemma~\ref{Lem: Reverse Hoelder for all balls} with $\alpha := \alpha_1$, we infer that~\eqref{Eq: Weak reverse Hoelder inequality} is valid for the operator $T_{\IR^d}$ and all balls $B(x_0 , r)$ with $0 < r < R_0$ and $B(x_0 , r) \cap \Omega \neq \emptyset$. Moreover, if $B(x_0 , r) \cap \Omega = \emptyset$,~\eqref{Eq: Weak reverse Hoelder inequality} is fulfilled trivially. Now, we distinguish the following two cases: \\
\textbf{Case 1: $R_0 = \infty$.} In this case,~\eqref{Eq: Weak reverse Hoelder inequality} is already verified for all balls in $\IR^d$ in the paragraph above this case. The proof can be concluded in this case. \\
\textbf{Case 2: $R_0 < \infty$.} In this case, we recall that~\eqref{Eq: Weak reverse Hoelder inequality} is valid for all balls with radius $0 < r < R_0$ and all balls with arbitrary radius and empty intersection with $\Omega$. Moreover, if $B(x_0 , r) \cap \Omega \neq \emptyset$ and
\begin{align*}
 r \geq \frac{\diam(\Omega)}{\alpha_2 - 1},
\end{align*}
we directly see that $\Omega \subset B(x_0 , \alpha_2 r)$. As a consequence, if a compactly supported function $f \in \L^{\infty} (\IR^d ; X)$ vanishes on $B(x_0 , \alpha_2 r)$, then $T_{\IR^d} f = 0$ by definition of $T_{\IR^d}$. Thus,~\eqref{Eq: Weak reverse Hoelder inequality} is valid for all balls that are large enough. Employing Lemma~\ref{Lem: Bridge from small balls to large balls} shows that~\eqref{Eq: Weak reverse Hoelder inequality} is valid for all balls in $\IR^d$. This concludes the proof.
\end{proof}

Estimates of type~\eqref{Eq: Weak reverse Hoelder inequality} and without the second term on the right-hand side are called \textit{weak reverse H\"older estimates}.

\begin{remark}
\label{Rem: Extension of operator families}
If $\T \subset \Lop(\L^2(\Omega ; X) , \L^2(\Omega ; Y))$ is a uniformly bounded operator family, we see that if we can verify the assumptions of Theorem~\ref{Thm: Extrapolation of square function estimates} with uniform constants for every operator in $\T$, the restriction of each operator to $\L^2(\Omega ; X) \cap \L^q(\Omega ; X)$ extends to a bounded operator from $\L^q(\Omega ; X)$ to $\L^q(\Omega ; Y)$, yielding a uniformly bounded family of operators.
\end{remark}

\section{Vector valued weak reverse H\"older estimates}
\label{Sec: Vector valued weak reverse Hoelder estimates}

We begin by proving Caccioppoli's inequality for higher-order elliptic systems subject to mixed boundary conditions. The proof is essentially the one of Barton~\cite[Sec.~3]{Barton}, with the modification that we not just consider balls, but also the sets $U_{y , r}^+$ defined in Remark~\ref{Rem: Mixed boundary geometry} and solutions which locally satisfy
\begin{align*}
 \lambda u_i + (-1)^m \sum_{j = 1}^N \sum_{\abs{\alpha} , \abs{\beta} = m} \partial^{\alpha} [\mu_{\alpha \beta}^{i j} \partial^{\beta} u_j] = 0 \qquad (1 \leq i \leq N).
\end{align*}
Barton considered only the case $\lambda = 0$. \par
As we are now concerned with the proof of Theorem~\ref{Thm: Maintheorem}, we will assume that $\Omega$ and $D$ are subject to Assumption~\ref{Ass: Mixed boundary geometry}. We will also use $\lambda$ as a resolvent parameter, so $\lambda \in \Sec_{\pi - \theta}$, where $\theta \in (\omega , \pi]$ and $\omega$ is such that $A$ is sectorial of angle $\omega$ on $\L^2(\Omega ; \IC^N)$. Recall that $M$ is the bound for the bi-Lipschitz constants of the homeomorphisms $\Phi_x$, see Assumption~\ref{Ass: Mixed boundary geometry}. Finally, we agree upon writing $\| \cdot \|_{\L^p(\Xi)}$ instead of $\| \cdot \|_{\L^p(\Xi ; \IC^l)}$ for sets $\Xi \subset \IR^d$ and $l \in \IN$.

\begin{lemma}[Caccioppoli's inequality part~1]
\label{Lem: Caccioppoli (1)}
Let $x_0 \in \overline{\Omega}$ and $r > 0$. Distinguish the following cases:
\begin{enumerate}
 \item[(1)] $x_0 \in \Omega$ and $r < \dist(x_0 , \partial \Omega)$;
 \item[(2)] $x_0 \in \partial \Omega$ with $\dist(x_0 , \overline{\partial \Omega \setminus D}) \leq 1 / (2 M)$ and $r \leq 1 / 4$;
 \item[(3)] $x_0 \in \partial \Omega$ with $\dist(x_0 , \overline{\partial \Omega \setminus D}) > 1 / (2 M)$ and $r \leq 1 / (2 M)$.
\end{enumerate}
Let $0 < s < t \leq 1$ and $f \in \L^2(\Omega ; \IC^N)$ be such that $f = 0$ on $B(x_0 , r) \cap \Omega$ in cases $(1)$ and $(3)$ or $f = 0$ on $U_{x_0 , r}^+$ in case $(2)$. Define $u := (\lambda + A)^{-1} f$. Then there exists a constant $C > 0$ depending only on $d$, $N$, $m$, $\kappa$, $M$, $\theta$, $\omega$, and $\Lambda$, such that
\begin{align*}
 \abs{\lambda} \int_{\mathcal{B}_{s r} \cap \Omega} \abs{u}^2 \; \d x &+ \int_{\mathcal{B}_{s r} \cap \Omega} \abs{\nabla^m u}^2 \; \d x \leq C \sum_{k = 0}^{m - 1} [(t - s) r]^{- 2(m - k)} \int_{[\mathcal{B}_{t r} \setminus \mathcal{B}_{s r}] \cap \Omega} \lvert \nabla^k u \rvert^2 \; \d x,
\end{align*}
where in cases $(1)$ and $(3)$, $\mathcal{B}_{\alpha r} := B(x_0 , \alpha r),$ and in case $(2)$, $\mathcal{B}_{\alpha r} := U_{x_0 , \alpha r}$ for $\alpha \in (0 , 1]$.
\end{lemma}

\begin{proof}
Take a cutoff function $\varphi \in \C_c^{\infty} (\IR^d)$ which in cases $(1)$ and $(3)$ is identically one on $B(x_0 , sr)$ and zero on $B(x_0 , tr)^c$ and which satisfies $\| \nabla^k \varphi \|_{\L^{\infty} (\IR^d)} \leq C_d [(t - s)r]^{-k}$ for all $0 \leq k \leq m$. In case $(2)$, take $\varphi$ to be one on $U_{x_0 , sr}$ and zero on $U_{x_0 , tr}^c$ with estimates $\| \nabla^k \varphi \|_{\L^{\infty} (\IR^d)} \leq C_{d , M} [(t - s)r]^{-k}$ for all $0 \leq k \leq m$. In case~(2), such a function $\varphi$ exists by the estimates proven in Remark~\ref{Rem: Mixed boundary geometry}~(1). The constant $C_d$ depends only on $d$ and $C_{d , M}$ on $d$ and $M$. \par
Define $\psi := u \varphi^{2m}$, which again is a function in $\W^{m , 2}_{\ID} (\Omega ; \IC^N)$, since $\varphi$ is smooth. Testing with $\psi$ yields
\begin{align*}
 0 = \lambda \int_{\mathcal{B}_{tr} \cap \Omega} \abs{u \varphi^m}^2 \; \d x + \mathfrak{a}(u , \psi).
\end{align*}
Moreover, Leibniz' rule yields numbers $c_{\alpha \beta} \in \IN$ with $c_{\alpha 0} = c_{\alpha \alpha} = 1$ such that
\begin{align*}
 \mathfrak{a}(u , \psi) &= \sum_{\abs{\alpha} , \abs{\beta} = m} \sum_{i , j = 1}^N \sum_{\gamma < \alpha} \int_{[\mathcal{B}_{tr} \setminus \mathcal{B}_{sr}] \cap \Omega} \mu_{\alpha \beta}^{i j} \partial^{\beta} u_j c_{\alpha \gamma} \partial^{\alpha - \gamma} \varphi^m \partial^{\gamma} [\varphi^m \overline{u_i}] \; \d x \\
 &\qquad+ \sum_{\abs{\alpha} , \abs{\beta} = m} \sum_{i , j = 1}^N \int_{\mathcal{B}_{tr} \cap \Omega} \mu_{\alpha \beta}^{i j} \partial^{\beta} u_j \varphi^m \partial^{\alpha} [\varphi^m \overline{u_i}] \; \d x.
\end{align*}
Note that the integration in the first integral on the right-hand side is performed only on $[\mathcal{B}_{tr} \setminus \mathcal{B}_{sr}] \cap \Omega$ since $\varphi^m$ is constant on both $\mathcal{B}_{sr}$ and $\mathcal{B}_{t r}^c$. Next, employing Leipniz' formula one can show that there exist smooth functions $\zeta_{\alpha \beta}$ such that
\begin{align}
\label{Eq: Auxiliary Leibniz formula}
 \sum_{\gamma < \alpha} c_{\alpha \gamma} \partial^{\alpha - \gamma} \varphi^m \partial^{\gamma} [\varphi^m \overline{u_i}] = \sum_{\delta < \alpha} \varphi^m \zeta_{\alpha \delta} \partial^{\delta} \overline{u_i}
\end{align}
and $\|\zeta_{\alpha \delta}\|_{\L^{\infty} (\IR^d)} \leq C_{d , M , m} [(t - s)r]^{\abs{\delta} - \abs{\alpha}}$, where $C_{d , M , m}$ solely depends on $d$, $M$, and $m$. Using~\eqref{Eq: Auxiliary Leibniz formula}, we derive
\begingroup
\allowdisplaybreaks
\begin{align*}
 \mathfrak{a}(u , \psi) &= \sum_{\abs{\alpha} , \abs{\beta} = m} \sum_{i , j = 1}^N \sum_{\delta < \alpha} \int_{[\mathcal{B}_{tr} \setminus \mathcal{B}_{sr}] \cap \Omega} \mu_{\alpha \beta}^{i j} \partial^{\beta} u_j \varphi^m \zeta_{\alpha \delta} \partial^{\delta} \overline{u_i} \; \d x \\
 &\qquad+ \sum_{\abs{\alpha} , \abs{\beta} = m} \sum_{i , j = 1}^N \int_{\mathcal{B}_{tr} \cap \Omega} \mu_{\alpha \beta}^{i j} \partial^{\beta} u_j \varphi^m \partial^{\alpha} [\varphi^m \overline{u_i}] \; \d x.
\intertext{Rewriting $\partial^{\beta} u_j \varphi^m$ by using Leibniz' rule reveals}
 &= - \sum_{\abs{\alpha} , \abs{\beta} = m} \sum_{i , j = 1}^N \sum_{\delta < \alpha} \sum_{\gamma < \beta} \int_{[\mathcal{B}_{tr} \setminus \mathcal{B}_{sr}] \cap \Omega} \mu_{\alpha \beta}^{i j} c_{\beta \gamma} \partial^{\gamma} u_j \partial^{\beta - \gamma} \varphi^m \zeta_{\alpha \delta} \partial^{\delta} \overline{u_i} \; \d x \\
&\qquad- \sum_{\abs{\alpha} , \abs{\beta} = m} \sum_{i , j = 1}^N \sum_{\gamma < \beta} \int_{[\mathcal{B}_{tr} \setminus \mathcal{B}_{sr}] \cap \Omega} \mu_{\alpha \beta}^{i j} c_{\beta \gamma} \partial^{\gamma} u_j \partial^{\beta - \gamma} \varphi^m \partial^{\alpha} [\varphi^m \overline{u_i}] \; \d x \\
 &\qquad+ \sum_{\abs{\alpha} , \abs{\beta} = m} \sum_{i , j = 1}^N \sum_{\delta < \alpha} \int_{[\mathcal{B}_{tr} \setminus \mathcal{B}_{sr}] \cap \Omega} \mu_{\alpha \beta}^{i j} \partial^{\beta} [\varphi^m u_j] \zeta_{\alpha \delta} \partial^{\delta} \overline{u_i} \; \d x \\
 &\qquad+ \sum_{\abs{\alpha} , \abs{\beta} = m} \sum_{i , j = 1}^N \int_{\mathcal{B}_{tr} \cap \Omega} \mu_{\alpha \beta}^{i j} \partial^{\beta} [\varphi^m u_j] \partial^{\alpha} [\varphi^m \overline{u_i}] \; \d x.
\end{align*}
\endgroup
Note that the last term on the right-hand side can be identified with $\mathfrak{a}(\varphi^m u , \varphi^m u)$. Summarizing, we find a constant $C >  0$ depending only on $d$, $N$, $m$, $M$, and $\Lambda$ such that
\begin{align*}
 \Big\lvert \lambda \int_{\mathcal{B}_{tr} \cap \Omega} &\abs{\varphi^m u}^2 \; \d x + \mathfrak{a}(\varphi^m u , \varphi^m u) \Big\rvert \\
 &\leq C \bigg\{ \sum_{\abs{\alpha} , \abs{\beta} = m} \sum_{\delta < \alpha} \sum_{\gamma < \beta} \frac{\|\partial^{\gamma} u\|_{\L^2([\mathcal{B}_{tr} \setminus \mathcal{B}_{sr}] \cap \Omega)}}{[(t - s) r]^{m - \abs{\gamma}}} \frac{\|\partial^{\delta} u\|_{\L^2([\mathcal{B}_{tr} \setminus \mathcal{B}_{sr}] \cap \Omega)}}{[(t - s) r]^{m - \abs{\delta}}} \\
 &\qquad+ \sum_{\abs{\beta} = m} \sum_{\gamma < \beta} \frac{\|\partial^{\gamma} u\|_{\L^2([\mathcal{B}_{tr} \setminus \mathcal{B}_{sr}] \cap \Omega)}}{[(t - s) r]^{m - \abs{\gamma}} } \| \nabla^m [\varphi^m u] \|_{\L^2(\mathcal{B}_{tr} \cap \Omega)} \bigg\}.
\end{align*}
Using the sectoriality of $\mathfrak{a}$, see Remark~\ref{Rem: Numerical range}, as well as $\lambda \in \Sec_{\pi - \theta}$ and $\pi - \theta + \omega < \pi$, we conclude that there exists a constant $C_{\theta , \omega}$ depending only on $\theta$ and $\omega$ such that
\begin{align*}
 \Big\lvert \lambda \int_{\mathcal{B}_{tr} \cap \Omega} \abs{\varphi^m u}^2 &\; \d x + \mathfrak{a}(\varphi^m u , \varphi^m u) \Big\rvert \\
 &\geq C_{\theta , \omega} \Big\{ \abs{\lambda} \int_{\mathcal{B}_{t r} \cap \Omega} \lvert \varphi^m u \rvert^2 \; \d x + \lvert \fa (\varphi^m u , \varphi^m u) \rvert \Big\} \\
 \intertext{holds. By G\r{a}rding's inequality, we derive}
 &\geq C_{\theta , \omega} \Big\{ \abs{\lambda} \int_{\mathcal{B}_{t r} \cap \Omega} \lvert \varphi^m u \rvert^2 \; \d x + \kappa \int_{\mathcal{B}_{t r} \cap \Omega} \lvert \nabla^m [\varphi^m u] \rvert^2 \; \d x \Big\}.
\end{align*}
Next, use Young's inequality to estimate
\begin{align*}
 C \sum_{\abs{\beta} = m} \sum_{\gamma < \beta} &\frac{\|\partial^{\gamma} u\|_{\L^2([\mathcal{B}_{tr} \setminus \mathcal{B}_{sr}] \cap \Omega)}}{[(t - s) r]^{m - \abs{\gamma}} } \| \nabla^m [\varphi^m u] \|_{\L^2(\mathcal{B}_{tr} \cap \Omega)} \\
 &\leq \frac{C}{2 \eps} \sum_{\abs{\beta} = m} \sum_{\gamma < \beta} \frac{\|\partial^{\gamma} u\|_{\L^2([\mathcal{B}_{tr} \setminus \mathcal{B}_{sr}] \cap \Omega)}^2}{[(t - s) r]^{2(m - \abs{\gamma})} } + \frac{C \eps}{2} \sum_{\abs{\beta} = m} \sum_{\gamma < \beta} \| \nabla^m [\varphi^m u] \|_{\L^2(\mathcal{B}_{tr} \cap \Omega)}^2.
\end{align*}
Choose $\eps$, such that
\begin{align*}
 \frac{C \eps}{2} \sum_{\abs{\beta} = m} \sum_{\gamma < \beta} 1 = \frac{C_{\theta , \omega} \kappa}{2}.
\end{align*}
Then, absorb $C_{\theta , \omega} \kappa \| \nabla^m [\varphi^m u] \|_{\L^2(\mathcal{B}_{tr} \cap \Omega)}^2 / 2$ from the right-hand side onto the left-hand side of the whole inequality. Using that and $\varphi = 1$ on $\mathcal{B}_{sr}$ concludes the proof.
\end{proof}

The preceding lemma shows that one can locally control $\abs{\lambda}^{1 / 2} u$ and $\nabla^m u$ in $\L^2$ by the $\L^2$-norms of all derivatives of order strictly less than $m$. However, it is desirable to control them solely by $u$ in the $\L^2$-norm. To prove that, we adapt the proof of Barton~\cite[Thm.~18]{Barton} to mixed boundary conditions. For this purpose, we prove the following lemma, which is a generalization of Giaquinta and Martinazzi~\cite[Lem.~8.18]{Giaquinta_Martinazzi} and is implicitly contained in the proof of Barton.

\begin{lemma}
\label{Lem: Giaquinta_Martinazzi lemma}
Let $0 \leq s_0 < t_0 < \infty$ and $k \in \IN$. Assume that $\phi : [s_0 , t_0] \to \IR$ is a non-negative bounded function. Suppose that there exist constants $A_1 , \dots , A_k > 0$, $\alpha_1 , \dots , \alpha_k > 0$, and $0 \leq \eps < 1$ such that for all $s_0 \leq s < t \leq t_0$ we have
\begin{align*}
 \phi (s) \leq \sum_{l = 1}^k A_l (t - s)^{- \alpha_l} + \eps \phi(t).
\end{align*}
Then there exists a constant $C > 0$ depending only on $\max_{i = 1}^k \{ \alpha_i \}$ and $\eps$, such that for all $s_0 \leq s < t \leq t_0$ we have
\begin{align*}
 \phi(s) \leq C \sum_{l = 1}^k A_l (t - s)^{- \alpha_l}.
\end{align*}
\end{lemma}

\begin{proof}
Let $0 < \tau < 1$ to be determined and define
\begin{align*}
 \rho_0 := s, \quad \rho_{n + 1} := \rho_n + (1 - \tau) \tau^n (t - s) \qquad (n \in \IN_0).
\end{align*}
Notice that
\begin{align*}
 \rho_{n + 1} = s + (1 - \tau) \sum_{j = 0}^n \tau^j (t - s) < s + (1 - \tau) \sum_{n = 0}^{\infty} \tau^j (t - s) = t.
\end{align*}
Deduce inductively
\begin{align*}
 \phi(\rho_0) \leq \sum_{l = 1}^k A_l (\rho_1 - \rho_0)^{- \alpha_l} + \eps \phi(\rho_1) \leq \sum_{j = 0}^{n - 1} \eps^j \sum_{l = 1}^k A_l (\rho_{j + 1} - \rho_j)^{- \alpha_l} + \eps^n \phi(\rho_n).
\end{align*}
Rearranging the sums on the right-hand side and using that $\rho_{j + 1} - \rho_j = (1 - \tau) \tau^j (t - s)$ yields
\begin{align*}
 \sum_{j = 0}^{n - 1} \eps^j \sum_{l = 1}^k A_l (\rho_{j + 1} - \rho_j)^{- \alpha_l} &= \sum_{l = 1}^k A_l (1 - \tau)^{- \alpha_l} (t - s)^{- \alpha_l} \sum_{j = 0}^{n - 1} \eps^j \tau^{- j \alpha_l}.
\end{align*}
Choose $\tau$ such that $\eps \tau^{- \max_i \{\alpha_i \}} < 1$ and let $n \to \infty$ to conclude
\begin{align*}
 \phi(s) &\leq (1 - \tau)^{- \max_i \{ \alpha_i \}} \sum_{j = 0}^{\infty} \big( \eps \tau^{- \max_i \{ \alpha_i \}} \big)^j \sum_{l = 1}^k A_l (t - s)^{- \alpha_l}. \qedhere
\end{align*}
\end{proof}

Now, we are ready to conclude the proof of Caccioppoli's inequality with the sole $\L^2$-norm of $u$ on the right-hand side. For the reduction of the differentiability on the right-hand side of the inequality in Lemma~\ref{Lem: Caccioppoli (1)}, recall that by Gagliardo--Nirenberg's inequality, one can estimate
\begin{align*}
 \| \nabla^{m - 1} u \|_{\L^2} \leq C \| u \|_{\L^2}^{\theta} \| u \|_{\W^{m , 2}}^{1 - \theta},
\end{align*}
for some $\theta \in (0 , 1)$.  The term involving $\| \nabla^m u \|_{\L^2}$ in the norm of $\| u \|_{\W^{m , 2}}$ can then be controlled by means of the first part of Caccioppoli's inequality, so that only terms of differentiability strictly less than $m$ occur on the right-hand side. Using Young's inequality, we can produce an $\eps$ in front of the $\L^2$-norm of $\nabla^{m - 1} u$ on the right-hand side. This leads to the situation of Lemma~\ref{Lem: Giaquinta_Martinazzi lemma}. \par
Due to the implicit dependence of the constants in the Gagliardo--Nirenberg inequality, we have to restrict the size of the parameter $s$ to be away from zero.

\begin{lemma}[Caccioppoli's inequality part~2]
\label{Lem: Caccioppoli (2)}
If in the situation of Lemma~\ref{Lem: Caccioppoli (1)} also $0 < 1 / 2 \leq s < t \leq 1$ holds, then there exists a constant $C > 0$ depending only on $d$, $N$, $m$, $\kappa$, $M$, $\theta$, $\omega$, and $\Lambda$ such that
\begin{align*}
 \int_{\mathcal{B}_{sr} \cap \Omega} \lvert \nabla^k u \rvert^2 \; \d x \leq \frac{C}{[(t - s) r]^{2k}} \int_{\mathcal{B}_{tr} \cap \Omega} \abs{u}^2 \; \d x \qquad (1 \leq k \leq m)
\end{align*}
holds.
\end{lemma}

\begin{proof}
We will prove the following claim by induction on $k$. Note that the initial step of this induction, i.e., $k = m$, is Lemma~\ref{Lem: Caccioppoli (1)} and that we will successively reduce the value of $k$. \\
\textbf{Claim:} There exists a constant $C > 0$ depending at most on $d$, $m$, $\kappa$, $N$, $M$, $\theta$, $\omega$, and $\Lambda$, such that for all $1 / 2 \leq s < t \leq 1$
\begin{align*}
 \int_{\mathcal{B}_{s r} \cap \Omega} \lvert \nabla^k u \rvert^2 \; \d x \leq C \sum_{l = 0}^{k - 1} [(t - s)r]^{-2 (k - l)} \int_{\mathcal{B}_{tr} \cap \Omega} \lvert\nabla^l u \rvert^2 \; \d x.
\end{align*}
\noindent First of all, we establish Gagliardo--Nirenberg's inequalities on the sets $\mathcal{B}_{\alpha r} \cap \Omega$ with $\alpha \in [1 / 2 , 1]$. As the constant of this inequality may depend on the size of the underlying sets, we rescale the whole situation. \par 
Rescale the function $u$ as $u_{\alpha r} : \frac{1}{\alpha r} (\mathcal{B}_{\alpha r} \cap \Omega) \to \IC^N , x \mapsto u(\alpha r x)$ and recall the cases presented in Lemma~\ref{Lem: Caccioppoli (1)}. Having a closer look onto $\frac{1}{\alpha r} (\mathcal{B}_{\alpha r} \cap \Omega),$ we see that in the first case this is just the ball $B([\alpha r]^{-1} x_0 , 1)$, which is a Sobolev extension domain of arbitrary order. In the third case, the set $\frac{1}{\alpha r} (\mathcal{B}_{\alpha r} \cap \Omega)$ is simply $B([\alpha r]^{-1} x_0 , 1) \cap \frac{1}{\alpha r} \Omega$. The radius $r$ is chosen such that $B(x_0 , r)$ only hits Dirichlet boundary so that $u_{\alpha r}$ can be identified with its extension by zero to $B( [\alpha r]^{-1} x_0 , 1)$. As above $B( [\alpha r]^{-1} x_0 , 1)$ is a Sobolev extension domain of all orders. Case $(2)$ is more interesting. Here, we have for some $x \in \overline{\partial \Omega \setminus D}$
\begin{align*}
 \frac{1}{\alpha r} (\mathcal{B}_{\alpha r} \cap \Omega) = \Phi_{x , \alpha r}^{-1} (Q( [\alpha r]^{-1} \Phi_x(x_0) , 1) \cap [\IR^{d - 1} \times (0 , \infty)]),
\end{align*}
where $\Phi_{x , \alpha r}^{-1}$ is given by $\frac{1}{\alpha r} \Phi_x^{-1} (\alpha r \cdot)$. Note that $\Phi_{x , \alpha r}^{-1}$ is a bi-Lipschitz homeomorphism, with bi-Lipschitz constant bounded by $M$ and that the bisected cube $Q( [\alpha r]^{-1} \Phi_x (x_0) , 1) \cap [\IR^{d - 1} \times (0 , \infty)]$ is an $(\eps , \delta)$-domain in the sense of Jones~\cite{Jones}, see Egert~\cite[Lem.~2.2.20]{Egert}. Moreover, a global bi-Lipschitz image of an $(\eps , \delta)$-domain is again an $(\eps , \delta)$-domain so that in our case, $\eps$ and $\delta$ only depend on $d$ and $M$. Finally, by Remark~\ref{Rem: Mixed boundary geometry}~(2)
\begin{align*}
 \diam( \Phi_{x , \alpha r}^{-1} (Q([\alpha r]^{-1} \Phi_x(x_0) , 1) \cap [\IR^{d - 1} \times (0 , \infty)])) \geq \frac{1}{M \sqrt{d}}
\end{align*}
so that the extension result of Rogers~\cite[Thm.~8]{Rogers} yields a Sobolev extension operator of arbitrary order with an operator norm depending only on $d$ and $M$. Extending in all three cases functions on $\frac{1}{\alpha r} (\mathcal{B}_{\alpha r} \cap \Omega)$ to $\IR^d$ and using the Gagliardo--Nirenberg inequalities on the whole space, see Nirenberg~\cite[p.~125]{Nirenberg}, proves that these inequalities hold true on $\frac{1}{\alpha r} (\mathcal{B}_{\alpha r} \cap \Omega)$ where the constant solely depends on $d$, $M$, and the Gagliardo--Nirenberg constants on $\IR^d$. Especially, we find for $\vartheta = \frac{1}{k + 1}$
\begin{align*}
 \| \nabla^k u_{\alpha r} \|_{\L^2(\frac{1}{\alpha r} (\mathcal{B}_{\alpha r} \cap \Omega))}^2 \leq C \|u_{\alpha r}\|_{\L^2(\frac{1}{\alpha r} (\mathcal{B}_{\alpha r} \cap \Omega))}^{2 \vartheta} \bigg[ \sum_{l = 0}^{k + 1} \int_{\frac{1}{\alpha r} (\mathcal{B}_{\alpha r} \cap \Omega)} \lvert \nabla^l u_{\alpha r} \rvert^2 \; \d x \bigg]^{1 - \vartheta}.
\end{align*}
By linear transformation, we find
\begin{align}
\label{Eq: Local Gagliardo-Nirenberg}
\begin{aligned}
 \| \nabla^k u \|_{\L^2(\mathcal{B}_{\alpha r} \cap \Omega)}^2 \leq C \big[ (\alpha r)^{-k} \|u\|_{\L^2(\mathcal{B}_{\alpha r} \cap \Omega)} \big]^{2 \vartheta} \bigg[ \sum_{l = 0}^{k + 1} (\alpha r)^{2(l - k)} \int_{\mathcal{B}_{\alpha r} \cap \Omega} \lvert \nabla^l u \rvert^2 \; \d x \bigg]^{1 - \vartheta}.
\end{aligned}
\end{align}
Next, let $1 / 2 \leq s < \tau \leq t \leq 1$ and apply~\eqref{Eq: Local Gagliardo-Nirenberg} on $\mathcal{B}_{s r} \cap \Omega$ as well as the induction hypothesis to the term involving $\nabla^{k + 1} u$, so that
\begin{align*}
 &\|\nabla^k u\|_{\L^2(\mathcal{B}_{s r} \cap \Omega)}^2 \\
 &\quad\leq C [ (s r)^{-k} \|u\|_{\L^2(\mathcal{B}_{s r} \cap \Omega)}]^{2 \vartheta} \bigg[ \sum_{l = 0}^k \big[ s^{2(l - k)} + s^2 (\tau - s)^{- 2(k + 1 - l)} \big] r^{2(l - k)} \int_{\mathcal{B}_{\tau r} \cap \Omega} \lvert \nabla^l u \rvert^2 \; \d x \bigg]^{1 - \vartheta}.
\intertext{For some $\delta > 0$, use Young's inequality $ab \leq \vartheta \delta^{(\vartheta - 1) / \vartheta} a^{\frac{1}{\vartheta}} + (1 - \vartheta) \delta b^{\frac{1}{1 - \vartheta}}$, to obtain}
 &\quad\leq \frac{C \delta^{\frac{\vartheta - 1}{\vartheta}}}{(s r)^{2k}} \| u \|_{\L^2(\mathcal{B}_{s r} \cap \Omega)}^2 + C (1 - \vartheta) \delta \sum_{l = 0}^k \big[ s^{2(l - k)} + s^2 (\tau - s)^{- 2(k + 1 - l)} \big] r^{2(l - k)} \int_{\mathcal{B}_{\tau r} \cap \Omega} \lvert \nabla^l u \rvert^2 \; \d x.
\end{align*}
Next, choose $\delta$ subject to the condition
\begin{align*}
 C \delta \frac{s^2}{(\tau - s)^2} = \frac{1}{8} \quad \Leftrightarrow \quad \delta = \frac{(\tau - s)^2}{8 C s^2}.
\end{align*}
Note that $\delta \leq 1 / (2 C)$, since $s \geq 1 / 2$ and $\tau - s \leq 1$. Thus,
\begin{align*}
 C \delta \big[ s^{2(l - k)} + s^2 (\tau - s)^{- 2(k + 1 - l)} \big] \leq  \frac{s^{2(l - k)}}{2} + \frac{(\tau - s)^{- 2(k - l)}}{8}
\end{align*}
and, by means of the choice $\vartheta = \frac{1}{k + 1}$,
\begin{align*}
 \frac{C \delta^{\frac{\vartheta - 1}{\vartheta}}}{s^{2 k}} = \frac{C}{s^{2 k} \delta^k} = \frac{8^k C^{1 + k}}{(\tau - s)^{2k}}.
\end{align*}
Returning to the estimate of $\nabla^k u$, estimate each $s$ from below by $\tau - s$ and use for all terms on the right-hand side of differentiability strictly less than $k$, that $\mathcal{B}_{s r} \cap \Omega$ and $\mathcal{B}_{\tau r} \cap \Omega$ are contained in $\mathcal{B}_{t r} \cap \Omega$. Put everything together to deduce
\begin{align*}
 \|\nabla^k u\|_{\L^2(\mathcal{B}_{s r} \cap \Omega)}^2 &\leq \Big[ 8^k C^{1 + k} + \frac{5 (1 - \vartheta)}{8} \Big] \frac{1}{[(\tau - s) r]^{2 k}} \| u \|_{\L^2(\mathcal{B}_{t r} \cap \Omega)}^2 \\
 &\qquad+ \frac{5 (1 - \vartheta)}{8} \sum_{l = 1}^{k - 1} \frac{1}{[(\tau - s) r]^{2 (k - l)}} \| \nabla^l u \|_{\L^2(\mathcal{B}_{t r} \cap \Omega)}^2 \\
 &\qquad+ \frac{5 (1 - \vartheta)}{8} \| \nabla^k u \|_{\L^2(\mathcal{B}_{\tau r} \cap \Omega)}^2.
\end{align*}
Now, we can appeal to Lemma~\ref{Lem: Giaquinta_Martinazzi lemma} by means of the following definitions. Let $s_0 := 1 / 2$, $t_0 := t$, $\alpha_l := 2(k - (l - 1))$,
\begin{align*}
 A_1 &:= \Big[ 8^k C^{1 + k} + \frac{5 (1 - \vartheta)}{8} \Big] \| u \|_{\L^2(\mathcal{B}_{t r} \cap \Omega)}^2, \\
 A_l &:= \frac{5 (1 - \vartheta)}{8} \| \nabla^{l - 1} u \|_{\L^2(\mathcal{B}_{t r} \cap \Omega)}^2 \qquad (2 \leq l \leq k),
\end{align*}
and
\begin{align*}
 \phi(s) := \| \nabla^k u \|_{\L^2(\mathcal{B}_{s r} \cap \Omega)}^2.
\end{align*}
It follows that there exists a constant $C > 0$ (different from the one above but independent of $s$, $\tau$, $r$, and $t$) such that for all $s \leq \tau \leq t$
\begin{align*}
 \| \nabla^k u \|_{\L^2(\mathcal{B}_{s r} \cap \Omega)}^2 \leq C \sum_{l = 0}^{k - 1} \frac{1}{[(\tau - s) r]^{2 (k - l)}} \| \nabla^l u \|_{\L^2(\mathcal{B}_{t r} \cap \Omega)}^2.
\end{align*}
In particular, this holds true for $\tau = t$, so that we conclude the induction step.
\end{proof}

The following lemma is a vector-valued and local version of the Sobolev embedding theorem, mentioned in the introduction.

\begin{lemma}
\label{Lem: Vector-valued Sobolev embedding}
Let $p \in [1 , \infty)$, $n_0 \in \IN$, and $(u_n)_{n = 1}^{n_0} \subset \W^{1 , p}_{\ID} (\Omega ; \IC^N)$. Let $x_0 \in \overline{\Omega}$ and $0 < r \leq 1 / (4 M \sqrt{d})$ be such that either $B(x_0 , r) \subset \Omega$ or $x_0 \in \partial \Omega$. If $q \in [1 , \infty)$ is such that $0 \leq \delta := \frac{1}{p} - \frac{1}{q} \leq \frac{1}{d}$, then there exists a constant $C > 0$ depending only on $p$, $q$, $d$, $N$, and $M$ such that
\begin{align*}
&\bigg( \frac{1}{r^d} \int_{B(x_0 , r) \cap \Omega} \Big[ \sum_{n = 1}^{n_0} \abs{u_n}^2 \Big]^{\frac{q}{2}} \; \d x \bigg)^{\frac{1}{q}} \\
 &\quad\leq C \bigg\{ r \bigg( \frac{1}{r^d} \int_{\mathcal{B} \cap \Omega} \Big[ \sum_{n = 1}^{n_0} \abs{\nabla u_n}^2 \Big]^{\frac{p}{2}} \; \d x \bigg)^{\frac{1}{p}} + \frac{1}{r^d} \int_{B(x_0 , r) \cap \Omega} \Big[ \sum_{n = 1}^{n_0} \abs{u_n}^2 \Big]^{\frac{1}{2}} \; \d x \bigg\}
\end{align*}
holds, where $\mathcal{B} = B(x_0 , M^2 \sqrt{d} r)$ if $x_0 \in \partial \Omega$ with $\dist(x_0 , \overline{\partial \Omega \setminus D}) \leq 1 / (2 M)$, and $\mathcal{B} = B(x_0 , r)$ else.
\end{lemma}

\begin{proof}
If $x_0 \in \partial \Omega$, extend a function $u \in \W^{1 , 2}_{\ID}(\Omega ; \IC^N)$ from $\Omega \cap B(x_0 , r)$ to $B(x_0 , r)$ componentwise as follows: If $\dist (x_0 , \overline{\partial \Omega \setminus D}) \leq 1 / (2 M)$, then $B(x_0 , r) \subset U_{x_0 , M \sqrt{d} r}$ by Remark~\ref{Rem: Mixed boundary geometry}~(2), so define $(\cE u)_i := \cE_{x_0 , M \sqrt{d} r} u_i$ by using the local extension operator $\cE_{x_0 , M \sqrt{d} r}$ given by Proposition~\ref{Prop: Local extension operators}. For all other $x_0 \in \partial \Omega$, let $\cE u$ denote the extension by zero. \par
We present the most difficult case where $\dist (x_0 , \overline{\partial \Omega \setminus D}) \leq 1 / (2 M)$ and point out changes in the proof for the other cases afterwards. Note that as $U_{x_0 , r/(M \sqrt{d})}^+$ is the bi-Lipschitz image of a set with measure comparable to $r^d$, it holds
\begin{align}
\label{Eq: (d - 1)-set estimate}
 \lvert U_{x_0 , r/(M \sqrt{d})}^+ \rvert \geq \mathcal{C} r^d
\end{align}
with $\mathcal{C}$ depending only on $d$ and $M$. By Remark~\ref{Rem: Mixed boundary geometry}~(2), we conclude that $\abs{B(x_0 , r) \cap \Omega} \geq \mathcal{C} r^d$ holds true. Use the triangle inequality to conclude
\begingroup
\allowdisplaybreaks
\begin{align*}
 \bigg( \int_{B(x_0 , r) \cap \Omega} \Big[ \sum_{n = 1}^{n_0} \abs{u_n}^2 \Big]^{\frac{q}{2}} \; \d x \bigg)^{\frac{1}{q}} &\leq \bigg( \int_{B(x_0 , r)} \Big[ \sum_{n = 1}^{n_0} \abs{\cE u_n (x) - (\cE u_n)_{B(x_0 , r) \cap \Omega}}^2 \Big]^{\frac{q}{2}} \; \d x \bigg)^{\frac{1}{q}} \\
 &\qquad+ [\abs{B(0 , 1)} r^d]^{\frac{1}{q}} \Big[ \sum_{n = 1}^{n_0} \abs{(u_n)_{B(x_0 , r) \cap \Omega}}^2 \Big]^{\frac{1}{2}}. \intertext{Next, $\lvert \cE u_n (x)- (\cE u_n)_{B(x_0 , r) \cap \Omega} \rvert \leq 2^d / (d \cC) \int_{B(x_0 , r)} \abs{x - y}^{1 - d} \abs{\nabla \cE u_n (y)} \; \d y$ follows by a combination of Gilbarg and Trudinger~\cite[Lem.~7.16]{Gilbarg_Trudinger} and~\eqref{Eq: (d - 1)-set estimate}. Apply this to the first term and apply Minkowski's inequality as well as~\eqref{Eq: (d - 1)-set estimate} to the second term on the right-hand side, to obtain}
 &\leq \frac{2^d}{d \mathcal{C}} \Big\| \int_{B(x_0 , r)} \abs{\cdot - y}^{1 - d} \Big[ \sum_{n = 1}^{n_0} \abs{\nabla \cE u_n (y)}^2 \Big]^{\frac{1}{2}} \; \d y \Big\|_{\L^q(B(x_0 , r))} \\
 &\qquad+ \frac{\abs{B(0 , 1)}^{\frac{1}{q}} r^{\frac{d}{q} - d}}{\mathcal{C}} \int_{B(x_0 , r) \cap \Omega} \Big[ \sum_{n = 1}^{n_0} \abs{u_n}^2 \Big]^{\frac{1}{2}} \; \d x.
\intertext{For the first term on the right-hand side use the boundedness of the Riesz potential, see~\cite[Lem.~7.12]{Gilbarg_Trudinger} for $\delta < 1 / d$ and Adams and Hedberg~\cite[Thm.~3.1.4 (b)]{Adams_Hedberg} for $\delta = 1 / d$, to get}
 &\leq \frac{2^d}{d \mathcal{C}} B r^{1 - d \delta} \Big\| \Big[ \sum_{n = 1}^{n_0} \abs{\nabla \cE u_n}^2 \Big]^{\frac{1}{2}} \Big\|_{\L^p(B(x_0 , r))} \\
 &\qquad+ \frac{\abs{B(0 , 1)}^{\frac{1}{q}} r^{\frac{d}{q} - d}}{\mathcal{C}} \int_{B(x_0 , r) \cap \Omega} \Big[ \sum_{n = 1}^{n_0} \abs{u_n}^2 \Big]^{\frac{1}{2}} \; \d x,
\end{align*}
\endgroup
where $B$ depends only on $d$, $p$, and $q$. Note that Proposition~\ref{Prop: Local extension operators} gives
\begin{align*}
 \Big\| \Big[ \sum_{n = 1}^{n_0} \abs{\nabla \cE u_n}^2 \Big]^{\frac{1}{2}} \Big\|_{\L^p(B(x_0 , r) \setminus \Omega)} &= \Big\| \Big[ \sum_{n = 1}^{n_0} \abs{\nabla [u_n \circ \psi] }^2 \Big]^{\frac{1}{2}} \Big\|_{\L^p(B(x_0 , r) \setminus \Omega)} \\
 &\leq C \Big\| \Big[ \sum_{n = 1}^{n_0} \abs{\nabla u_n}^2 \Big]^{\frac{1}{2}} \Big\|_{\L^p(U_{x_0 , M \sqrt{d} r}^+)},
\end{align*}
where $C > 0$ depends only on $d$ and $M$. By Remark~\ref{Rem: Mixed boundary geometry}~(2), we conclude the proof for $x_0 \in \partial \Omega$ with $\dist (x_0 , \overline{\partial \Omega \setminus D}) \leq 1 / (2 M)$. \par
If $B(x_0 , r) \subset \Omega$, we do exactly the same without the extension operator. If $x_0 \in \partial \Omega$ with $\dist (x_0 , \overline{\partial \Omega \setminus D}) > 1 / (2 M)$, we proceed as above with the one exception that in the first inequality below~\eqref{Eq: (d - 1)-set estimate}, we introduce $(\cE u_n)_{B(x_0 , r)}$ instead of $(\cE u_n)_{B(x_0 , r) \cap \Omega}$. This has the effect of avoiding the need of an estimate of the form $\lvert B(x_0 , r) \cap \Omega \rvert \geq \cC r^d$, which is not available under the given geometric setup.
\end{proof}

The following lemma is an iterated version of the previous one.

\begin{lemma}
\label{Lem: Iterated vector-valued Sobolev embedding}
Let $p \in [1 , \infty)$, $n_0 \in \IN,$ and $(u_n)_{n = 1}^{n_0} \subset \W^{m , p}_{\ID} (\Omega ; \IC^N)$. Let $x_0 \in \overline{\Omega}$ and $0 < r \leq [M^2 \sqrt{d}]^{1 - m} / (4 M \sqrt{d})$ be such that either $B(x_0 , r) \subset \Omega$ or $x_0 \in \partial \Omega$. If $q \in [1 , \infty)$ satisfies $0 \leq \delta := \frac{1}{p} - \frac{1}{q} \leq \frac{m}{d}$, then there exists a constant $C > 0$ depending at most on $p$, $q$, $d$, $N$, $m$, and $M$ such that
\begin{align*}
\bigg( \frac{1}{r^d} \int_{B(x_0 , r) \cap \Omega} \Big[ \sum_{n = 1}^{n_0} \abs{u_n}^2 \Big]^{\frac{q}{2}} \; \d x \bigg)^{\frac{1}{q}} &\leq C \bigg\{ r^m \bigg( \frac{1}{r^d} \int_{\mathcal{B}_m \cap \Omega} \Big[ \sum_{n = 1}^{n_0} \abs{\nabla^m u_n}^2 \Big]^{\frac{p}{2}} \; \d x \bigg)^{\frac{1}{p}} \\
 &\qquad+ \sum_{k = 0}^{m - 1} \frac{1}{r^{d - k}} \int_{\mathcal{B}_k \cap \Omega} \Big[ \sum_{n = 1}^{n_0} \lvert \nabla^k u_n \rvert^2 \Big]^{\frac{1}{2}} \; \d x \bigg\}
\end{align*}
holds, where $\mathcal{B}_k := B(x_0 , [ M^2 \sqrt{d} ]^k r)$ if $x_0 \in \partial \Omega$ with $\dist(x_0 , \overline{\partial \Omega \setminus D}) \leq 1 / (2 M)$ and $\mathcal{B}_k := B(x_0 , r)$ else.
\end{lemma}

\begin{proof}
We will only consider the case where $x_0 \in \partial \Omega$ satisfies the inequality $\dist(x_0 , \overline{\partial \Omega \setminus D}) \leq 1 / (2 M)$. The other cases are proven in the same way, with the one exception, that the domain of integration stays the same in each iteration step. Define $p_1 \in [p , \infty)$ via the equation
\begin{align*}
 \frac{1}{p} - \frac{1}{p_1} = \frac{\delta}{m}.
\end{align*}
It follows that
\begin{align*}
 \frac{1}{p_1} = \frac{1}{p} - \frac{\delta}{m} \geq \frac{1}{p} - \frac{1}{p} + \frac{1}{q},
\end{align*}
so that in fact $p_1 \in [p , q]$. Inductively, define for $2 \leq k \leq m$ the number $p_k \in [p_{k - 1} , q]$ via
\begin{align*}
 \frac{1}{p_{k - 1}} - \frac{1}{p_k} = \frac{\delta}{m}.
\end{align*}
In particular, one finds $p_m = q$, so that Lemma~\ref{Lem: Vector-valued Sobolev embedding} can be applied iteratively. Thus,
\begin{align*}
 \bigg( \frac{1}{r^d} &\int_{B(x_0 , r) \cap \Omega} \Big[ \sum_{n = 1}^{n_0} \abs{u_n}^2 \Big]^{\frac{q}{2}} \; \d x \bigg)^{\frac{1}{q}} \\
 &\leq C \bigg\{ r \bigg( \frac{1}{r^d} \int_{\mathcal{B}_1 \cap \Omega} \Big[ \sum_{n = 1}^{n_0} \abs{\nabla u_n}^2 \Big]^{\frac{p_{m - 1}}{2}} \; \d x \bigg)^{\frac{1}{p_{m - 1}}} + \frac{1}{r^d} \int_{\mathcal{B}_0 \cap \Omega} \Big[ \sum_{n = 1}^{n_0} \abs{u_n}^2 \Big]^{\frac{1}{2}} \; \d x \bigg\} \\
 &\leq C \bigg\{ r^2 \bigg( \frac{1}{r^d} \int_{\mathcal{B}_2 \cap \Omega} \Big[ \sum_{n = 1}^{n_0} \lvert \nabla^2 u_n \rvert^2 \Big]^{\frac{p_{m - 2}}{2}} \; \d x \bigg)^{\frac{1}{p_{m - 2}}} \\
 &\qquad+ \frac{1}{r^{d - 1}} \int_{\mathcal{B}_1 \cap \Omega} \Big[ \sum_{n = 1}^{n_0} \abs{\nabla u_n}^2 \Big]^{\frac{1}{2}} \; \d x + \frac{1}{r^d} \int_{\mathcal{B}_0 \cap \Omega} \Big[ \sum_{n = 1}^{n_0} \abs{u_n}^2 \Big]^{\frac{1}{2}} \; \d x \bigg\}.
\intertext{Inductively, it follows}
 &\leq C \bigg\{ r^m \bigg( \frac{1}{r^d} \int_{\mathcal{B}_m \cap \Omega} \Big[ \sum_{n = 1}^{n_0} \abs{\nabla^m u_n}^2 \Big]^{\frac{p}{2}} \; \d x \bigg)^{\frac{1}{p}} + \sum_{k = 0}^{m - 1} \frac{1}{r^{d - k}} \int_{\mathcal{B}_k \cap \Omega} \Big[ \sum_{n = 1}^{n_0} \lvert \nabla^k u_n \rvert^2 \Big]^{\frac{1}{2}} \; \d x \bigg\}. \qedhere
\end{align*}

\end{proof}

Now we are in the position to prove the vector-valued weak reverse H\"older estimates.

\begin{theorem}
\label{Thm: Weak reverse Hoelder estimates}
Let $n_0 \in \IN$, $(\lambda_n)_{n = 1}^{n_0} \subset \Sec_{\pi - \theta}$, $x_0 \in \overline{\Omega}$, and $0 < r \leq 1 / (8 [M^2 \sqrt{d}]^{m + 1})$ be such that either $B(x_0 , 3 M [M^2 \sqrt{d}]^{m + 1} r) \subset \Omega$ or $x_0 \in \partial \Omega$. Moreover, let $(f_n)_{n = 1}^{n_0} \subset \L^2(\Omega ; \IC^N)$ such that for every $1 \leq n \leq n_0$ the function $f_n$ vanishes on $B(x_0 , 3 M [M^2 \sqrt{d}]^{m + 1} r) \cap \Omega$. If $2 m \geq d$, let $q$ be any number larger than $2$ and if $2 m < d$, let $q = \frac{2d}{d - 2m}$. Then there exists a constant $C > 0$, depending at most on $d$, $N$, $q$, $\kappa$, $m$, $M$, $\theta$, $\omega$, and $\Lambda$ such that
\begin{align*}
\bigg( \frac{1}{r^d} \int_{B(x_0 , r) \cap \Omega} \Big[ \sum_{n = 1}^{n_0} [ \abs{\lambda_n} \abs{u_n}]^2 \Big]^{\frac{q}{2}} \; \d x \bigg)^{\frac{1}{q}} \leq C \bigg( \frac{1}{r^d} \int_{B(x_0 , 2 [M^2 \sqrt{d}]^{m + 1} r) \cap \Omega} \sum_{n = 1}^{n_0} [ \abs{\lambda_n} \abs{u_n}]^2 \; \d x \bigg)^{\frac{1}{2}}
\end{align*}
holds.
\end{theorem}

\begin{proof}
Apply Lemma~\ref{Lem: Iterated vector-valued Sobolev embedding} with $q$ and $p = 2$ to the functions $(\lambda_n u_n)_{n = 1}^{n_0}$ to get
\begin{align*}
 \bigg( \frac{1}{r^d} \int_{B(x_0 , r) \cap \Omega} &\Big[ \sum_{n = 1}^{n_0} [ \abs{\lambda_n} \abs{u_n}]^2 \Big]^{\frac{q}{2}} \; \d x \bigg)^{\frac{1}{q}} \\
 &\leq C \bigg\{ r^m \bigg( \frac{1}{r^d} \int_{B(x_0 , [M^2 \sqrt{d}]^m r) \cap \Omega} \sum_{n = 1}^{n_0} [\abs{\lambda_n} \abs{\nabla^m u_n}]^2 \; \d x \bigg)^{\frac{1}{2}} \\
 &\qquad+ \sum_{k = 0}^{m - 1} \frac{1}{r^{d - k}} \int_{B(x_0 , [M^2 \sqrt{d}]^k r) \cap \Omega} \Big[ \sum_{n = 1}^{n_0} [\abs{\lambda_n} \lvert \nabla^k u_n \rvert]^2 \Big]^{\frac{1}{2}} \; \d x \bigg\}.
\intertext{Apply H\"older's inequality to the second term on the right-hand side, so that with a different constant $C$}
 &\leq C \sum_{k = 0}^m r^k \bigg( \frac{1}{r^d} \int_{B(x_0 , [M^2 \sqrt{d}]^k r) \cap \Omega} \sum_{n = 1}^{n_0} [\abs{\lambda_n} \lvert \nabla^k u_n \rvert]^2 \; \d x \bigg)^{\frac{1}{2}}.
\intertext{If $x_0$ is either inside $\Omega$ or on $\partial \Omega$ with $\dist(x_0 , \overline{\partial \Omega \setminus D}) > 1 / (2 M)$, then $2 [M^2 \sqrt{d}]^m r \leq 1 / (4 M^2 \sqrt{d})$, so that Lemma~\ref{Lem: Caccioppoli (2)} is directly applicable with $s = 1 / 2$, $t = 1$, and radius $2 [M^2 \sqrt{d}]^k r$. This yields}
&\leq C \bigg( \frac{1}{r^d} \int_{B(x_0 , 2 [M^2 \sqrt{d}]^m r) \cap \Omega} \sum_{n = 1}^{n_0} [\abs{\lambda_n} \abs{u_n}]^2 \; \d x \bigg)^{\frac{1}{2}}.
\end{align*}
In order to employ Caccioppoli's inequality in the case $\dist(x_0 , \overline{\partial \Omega \setminus D}) \leq 1 / (2M)$, apply Remark~\ref{Rem: Mixed boundary geometry}~(2), which in the current situation reads as
\begin{align*}
 B(x_0 , [M^2 \sqrt{d}]^k r) \subset U_{x_0 , M \sqrt{d} [M^2 \sqrt{d}]^k r} \subset B(x_0 , [M^2 \sqrt{d}]^{k + 1} r).
\end{align*}
Since $2 [M^2 \sqrt{d}]^{m + 1} r \leq 1 / 4$, Caccioppoli's inequality is applicable on the sets $\Omega \cap U_{x_0 , 2 [M^2 \sqrt{d}]^{k + 1} r}$ with $s = 1 / 2$ and $t = 1$. Proceeding as above concludes the proof.
\end{proof}

\subsection{Proof of Theorem~\ref{Thm: Maintheorem}}

By Remark~\ref{Rem: Reduction to p > 2} it suffices to concentrate on the case $p > 2$. Thus, in the case $2m < d$, define $p := 2d / (d - 2m)$, and in the case $2m \geq d$, let $p > 2$ be arbitrary. \par
As discussed in Section~\ref{Subsec: The operator}, the $\L^2$-realization of the elliptic operator $A$ is sectorial of angle $\omega \in [0 , \frac{\pi}{2})$. Thus, for every $\theta \in (\omega , \pi]$ the family $\{\lambda (\lambda + A)^{-1}\}_{\lambda \in \Sec_{\pi - \theta}}$ is bounded, which is equivalent to the boundedness of the family $\T$ given by
\begin{align*}
 \{ (\lambda_1 (\lambda_1 + A)^{-1} , \dots , \lambda_{n_0} (\lambda_{n_0} + A)^{-1} , 0 , \dots) : n_0 \in \IN , (\lambda_n)_{n = 1}^{n_0} \subset \Sec_{\pi - \theta} \}
\end{align*}
on $\L^2(\Omega ; \ell^2(\IC^N))$. With regard to Theorem~\ref{Thm: Extrapolation of square function estimates}, fix $X = Y = \ell^2(\IC^N)$ and invoke Remark~\ref{Rem: Extension of operator families} and Proposition~\ref{Prop: Reinterpretation of R-boundedness} to conclude that for each operator $T \in \T$ one has to verify the weak reverse H\"older estimates with uniform constants. Thus, we have to show that there exist constants $C > 0$, $R_0 > 0$, $\alpha_2 > \alpha_2 > 1$ such that for all $n_0 \in \IN$, $(\lambda_n)_{n = 1}^{n_0} \subset \Sec_{\pi - \theta}$, $(f_n)_{n = 1}^{n_0} \subset \L^{\infty} (\Omega ; \IC^N)$ with $f_n = 0$ on $B(x_0 , \alpha_2 r) \cap \Omega$, and $u_n := (\lambda_n + A)^{-1} f_n$ the estimate
\begin{align*}
\bigg( \frac{1}{r^d} \int_{B(x_0 , r) \cap \Omega} \Big[ \sum_{n = 1}^{n_0} [ \abs{\lambda_n} &\abs{u_n}]^2 \Big]^{\frac{p}{2}} \; \d x \bigg)^{\frac{1}{p}} \leq C \bigg( \frac{1}{r^d} \int_{B(x_0 , \alpha_1 r) \cap \Omega} \sum_{n = 1}^{n_0} [ \abs{\lambda_n} \abs{u_n}]^2 \; \d x \bigg)^{\frac{1}{2}}
\end{align*}
holds. This is exactly the statement of Theorem~\ref{Thm: Weak reverse Hoelder estimates}. Furthermore, Lemma~\ref{Lem: Reverse Hoelder for all balls} shows that one can take $\alpha_1 = 2$ and that the weak reverse H\"older estimates are valid for every $x_0$ satisfying $B(x_0 , r) \cap \Omega \neq \emptyset$. If $2m < d$, this allows us to invoke the self-improving property of weak reverse H\"older estimates, see~\cite[Thm.~6.38]{Giaquinta_Martinazzi}. Hence, there exists $\eps > 0$ such that
\begin{align*}
\bigg( \frac{1}{r^d} \int_{B(x_0 , r) \cap \Omega} \Big[ \sum_{n = 1}^{n_0} [ \abs{\lambda_n} &\abs{u_n}]^2 \Big]^{\frac{p + \eps}{2}} \; \d x \bigg)^{\frac{1}{p + \eps}} \leq C \bigg( \frac{1}{r^d} \int_{B(x_0 , 2 r) \cap \Omega} \Big[ \sum_{n = 1}^{n_0} [ \abs{\lambda_n} \abs{u_n}]^2 \Big]^{\frac{p}{2}} \; \d x \bigg)^{\frac{1}{p}}.
\end{align*}
We conclude that $\{ \lambda (\lambda + A_p)^{-1} \}_{\lambda \in \Sec{\pi - \theta}}$ is $\R$-bounded in $\Lop(\L^p(\Omega ; \IC^N))$. By~\cite[Prop.~2.1.1 h)]{Haase} we infer that $A_p$ is densely defined. \qed


\end{document}